\documentclass[11pt,reqno]{amsart}
\usepackage{hyperref}

\usepackage{amsmath, amssymb, latexsym, amscd, amsthm, amsfonts, amstext}
\usepackage[mathscr]{eucal}
\usepackage{xspace}

\voffset= - 0.5in \hoffset= 0.18 in
\theoremstyle{plain}
\newtheorem{theorem}{Theorem}[section]

\newtheorem{lemma}[theorem]{Lemma}

\newtheorem{proposition}[theorem]{Proposition}

\theoremstyle{definition}
\newtheorem{definition}[theorem]{Definition}

\begin{document}

\title[]{Blow up property for viscoelastic
evolution equations on manifolds with conical degeneration}

\centerline{}
\author[M. Alimohammady]{Mohsen Alimohammady}
\address{Department of Mathematics, Faculty of
Mathematical Sciences University of Mazandaran, Babolsar 47416-1468,
Iran} \email{\href{mailto: <amohsen@umz.ac.ir>}{amohsen@umz.ac.ir}}
\author[M. K. kalleji]{Morteza  Koozehgar kalleji}
\address{Department of Mathematics, Faculty of
Mathematical Sciences University of Mazandaran, Babolsar 47416-1468,
Iran}
\email{\href{mailto:<m.kalleji@yahoo.com>}{m.kalleji@yahoo.com}}

\keywords{ Viscoelastic equation, blow up, Cone Sobolev spaces,
degenerated 
 differential Operator.\\
$^\ast$ Corresponding author \\
\hspace*{.3cm} {\it 2010 Mathematics Subject Classifications}:
35L71,  74DXX, 35B44, 58JXX ·}

\date{\bf }

\begin{abstract}
This paper is concerned with the study of the nonlinear viscoelastic
evolution equation with strong damping and source terms, described
by \[u_{tt} - \Delta_{\mathbb{B}}u +
\int_{0}^{t}g(t-\tau)\Delta_{\mathbb{B}}u(\tau)d\tau +
f(x)u_{t}|u_{t}|^{m-2} = h(x)|u|^{p-2}u , \hspace{1 cm} x\in
int\mathbb{B}, t
>
0,\] where $\mathbb{B}$ is a stretched manifold. First, we prove the
solutions of problem {1.1} in cone Sobolev space
$\mathcal{H}^{1,\frac{n}{2}}_{2,0}(\mathbb{B}),$ admit a blow up in
finite time for $p > m$ and positive initial energy. Then, we
construct a lower bound for obtained blow up time under appropriate
assumptions on data.

\end{abstract}

\maketitle

\numberwithin{equation}{section}
\section{Introduction}

Nonlinear differential equations and their solutions play an
important role at description of physical and other processes. For
instance, a purely elastic material has a capacity to store
mechanical energy with no dissipation of the energy. A complete
opposite to an elastic material is a purely viscous material. The
important point about viscous materials is that when the force is
removed it does not return to its original shape. Materials which
are outside the scope of these two models will be those for which
some of the work done to deform them can be recovered. Such
materials possess a capacity of storage and dissipation of
mechanical energy. This is the case for viscoelastic material. It is
well known that viscoelastic materials exhibit natural damping,
which is due to the special property of these materials to retain a
memory of their past history \cite{FM,RHN}. It is worth mentioning
the works in connection with viscoelastic effects on a domain
$\Omega$ of $\mathbb{R}^{n}.$ Cavalcanti et al. \cite{Cav1}, firstly
investigated
\begin{equation}
\left\{\begin{array}{ll} u_{tt} - \Delta u +
\int_{0}^{t}g(t-\tau)\Delta u(\tau)d\tau + a(x)u_{t} + |u|^{\gamma}u
=0 , \hspace{1 cm} x\in \Omega,~~ t
>
0,\\
u(0,x) = u_{0}(x),~~~~~~u_{t}(0,x) = u_{1}(x) , \hspace{4.2 cm} x\in \Omega\\
u(t,x) = 0 ,\hspace{7.9 cm}x\in
\partial\Omega,~~t\geq 0,
\end{array}\right.
\end{equation}
and obtained an exponential decay rate of the solution under some
assumptions on $g(s)$ and $a(x).$ It is important to mention some
papers in connection with viscoelastic effects, among them, Alves
and Cavalcanti \cite{AlC}, Aassila et.al \cite{ACS} and references
therein. Cavalcanti and Oquendo \cite{CO} studied \[u_{tt} -
k_{0}\Delta u + \int_{0}^{t} div [a(x) g(t-\tau)\nabla u ]d\tau +
b(x) h(u_{t})  + f(u) = 0,\] under the restrictive assumptions on
both the damping function $h$ and the kernel $g.$ Moreover,
Messaoudi \cite{Mes1} studied the global existence of solutions for
the viscoelastic equation, at the same time he also obtained a
blow-up result with negative energy. Then, he improved his blow-up
result in \cite{Mes2}. Recently, Song and Xue \cite{SX}, considered
the nonlinear viscoelastic equation \[u_{tt} - \Delta u +
\int_{0}^{t} g(t - \tau) \Delta u(\tau)d\tau - \Delta u_{t} =
|u|^{p-2}u \hspace{1 cm}  on ~~\Omega\times [0,T],\] where $\Omega$
is a bounded domain of $\mathbb{R}^{n}$ with a smooth boundary
$\partial\Omega.$ They showed under suitable conditions on $g,$ that
there are solutions to their problem with arbitrarily high initial
energy which blow up in finite time. Cavalcanti et al. \cite{Cav2}
considered a nonlinear viscoelastic evolution equation as \[u_{tt} +
Au + F(x,t,u,u_{t}) - \int_{0}^{t}g(t - \tau)Au(\tau)d\tau =0
\hspace{1 cm} on~~\Gamma\times(0,\infty)\] where $\Gamma$ is a
compact manifold. When $F\neq 0$ and $g\neq 0$ they proved the
existence of global solutions as well as uniform decay rates.
Furthermore, Cavalcanti et al. \cite{Cav3}, studied the wave
equation on compact surfaces and locally distributed damping,
described by \[u_{tt} - \Delta_{\mathcal{M}}u + a(x) g(u_{t}) = 0
\hspace{1 cm} on ~\mathcal{M}\times (0,\infty),\] where $\mathcal{M}
\subset \mathbb{R}^{3}$ is a smooth oriented embedding compact
surface without boundary and $\Delta_{\mathcal{M}}$ is the
Laplace-Beltrami operator on $\mathcal{M}.$ Recently, the authors in
\cite{Cav4}, discussed the asymptotic stability of the wave equation
on a compact Riemannian manifold subject to locally distributed
viscoelastic effects. Also, they proved that the solutions of the
corresponding partial viscoelastic model decay exponentially to
zero. This paper motivated by two works. The first, analog to
\cite{Cav3}, we consider Fuchsian-Laplace operator
$\Delta_{\mathbb{B}}$ as a special case of the typical differential
operators on a manifold with conical singularity so-called Fuchs
type operator. The second, using of the nonlinear viscoelastic
evolution equation in \cite{Cav2}, we will investigate a nonlinear
viscoelastic evolution problem includes Fuchsian-Laplace operator,
strong damping and source terms. More precisely, we want to study
blow up of solutions a nonlinear viscoelastic problem contains
Fuchsian-Laplace operator and nonlinear damping and source terms on
a manifold with conical singularity points. Moreover, when blow up
occurs, the blow up time $T$ cannot usually be computed exactly. We
note that, in general, this is very difficult to obtain a lower
bound for viscoelastic wave equation. Hence, it is very attractive
and interesting subject of inquiry to specify lower  and upper
bounds for blow up time $T.$ From the mathematical point of view,
the damping effects are modeled by intro-differential operators.
Therefore, the dynamics of viscoelastic materials are of great
importance and interest as they have wide applications in natural
sciences. Thereupon, we consider an initial-boundary value problem
for a nonlinear viscoelastic wave equation with strong damping,
nonlinear damping and source terms on a manifold with conical
degenerations as follows:
\begin{equation}\label{1.1}
\left\{\begin{array}{ll} u_{tt} - \Delta_{\mathbb{B}}u +
\int_{0}^{t}g(t-\tau)\Delta_{\mathbb{B}}u(\tau)d\tau +
f(x)u_{t}|u_{t}|^{m-2} = h(x)|u|^{p-2}u , \hspace{1 cm} x\in
int\mathbb{B},~ t
>
0,\\
u(x,0) = u_{0}(x),~~~~~~u_{t}(x,0) = u_{1}(x) , \hspace{6.5 cm} x\in int\mathbb{B}\\
u(t,x) = 0 ,\hspace{10 cm}x\in
\partial\mathbb{B},~t\geq 0,
\end{array}\right.
\end{equation}
where $p>2,$ $m \geq 2 ,$ and $g$ is a $C^{1}-$function and we
assume that the function $g$ satisfies the following conditions:

$A_{1})$ $1 - \int_{0}^{\infty} g(s)ds = l > 0.$ This assumption
guarantee the hyperbolicity and well-posedness of the system
\ref{1.1}.

$A_{2})$ $g(s) \geq 0$ and $g'(s) \leq 0 ;$

$A_{3})$ $\int_{0}^{\infty} g(s) ds < 1 - \frac{1}{(p - 1)^{2}} ;$

$A_{4})$ $f,h\in L^{\infty}(int\mathbb{B})\cap C(int\mathbb{B})$ are
positive functions.

 The domain
$\mathbb{B}$ equals $[0,1) \times X,$ $X$ is an $(n -
1)$-dimensional closed compact manifold $(n \geq 3),$ which is
regarded as the local model near the conical points on manifolds
with conical singularities, and $\partial\mathbb{B}= \{0\} \times
X.$ Also, we assume that the volume of $\mathbb{B}$ is finite, i.e.
$|\mathbb{B}| = \int_{\mathbb{B}} \frac{dx_{1}}{x_{1}}dx' < +
\infty.$ Moreover, the operator $\Delta_{\mathbb{B}}$ in $\ref{1.1}$
is defined by $(x_{1}\partial_{x_{1}})^{2} +
\partial^{2}_{x_{2}}+ ...+
\partial^{2}_{x_{n}},$ which is an elliptic operator with totally
characteristic degeneracy on the boundary $x_{1} = 0,$ we also call
it Fuchsian type Laplace operator, and the corresponding gradient
operator by $\nabla_{\mathbb{B}}:= (x_{1}\partial_{x_{1}},
\partial_{x_{2}},..., \partial_{x_{n}}).$ Near $\partial\mathbb{B}$
we will often use coordinates $(x_{1} , x')= (x_{1},
x_{2},...,x_{n})$ for $x_{1}\in [0,1)$ and $x' \in X.$ To more
details about the manifold with singularities and elliptic
differetial operators with totally characteristic degeneration see
Section 2 and the references \cite{CLW1,S}.

In the absence of the viscoelastic term ($g = 0$,) the problem
\ref{1.1} has been extensively studied and results concerning
existence and nonexistence have been established in Euclidean
domains. In bounded domains, for the equation
\[u_{tt} - \Delta u +  a|u_{t}|^{m-2}u_{t} = b|u|^{p-2}u \hspace{1 cm} on~~\Omega \times (0,\infty),\]
 where $m \geq 2,$ $ p > 2$ and $a,b$ are positive constants, the
 interaction between the damping and source terms was first
 considered by Levine \cite{Le1,Le2} for the linear damping case $m =2$
 and also the author proved that solutions with negative initial
 energy blow up in finite time. Georgiev and Todorova
 \cite{GT} investigated Levine's results for the nonlinear damping
 case $m > 2$ and they proved that solutions with any initial data
 continue to exist globally if $ m \geq p$ and blow up in finite
 time if $ p > m$ and the initial energy is sufficiently negative.
 For the problem \ref{1.1} without viscoelastic and damping terms on a manifold  with conical
 singularity, the authors \cite{AK} used the cone Sobolev and
 Poincar$\acute{e}$
 inequalities to prove the existence results of global solutions
 with exponential decay and showed the blow up in finite time of
 solutions in cone Sobolev space $\mathcal{H}^{1,\frac{n}{2}}_{2,0}(\mathbb{B}).$

Suppose that $C_{emb}$ is the best constant of the Sobolev embedding
$\mathcal{H}^{1,\frac{n}{2}}_{2,0}\hookrightarrow L^{\frac{n}{p}}$
for $p\in [2, 2^{*}]$ and
\begin{eqnarray}\label{1.3}C_{*}:=
\frac{C_{emb}C_{Poin}}{l^{\frac{1}{2}}}, \end{eqnarray} where
$C_{Poin}$ is the Poincar\'{e} constant. Furthermore, we assume that
\begin{eqnarray}\label{1}
C^{*} = \inf\Biggl\{\frac{\int_{\mathbb{B}}f(x)|u_{t}|^{m}
\frac{dx_{1}}{x_{1}}dx'}{\|\nabla_{\mathbb{B}}u\|^{m}_{L^{\frac{n}{m}}(\mathbb{B})}}~~~~~~~~~~:~~~~~~~~~u\in
\mathcal{H}^{1,\frac{n}{2}}_{2,0}(\mathbb{B})\Biggr\},
\end{eqnarray} and
\begin{eqnarray}\label{2}
\max ~~\{~m,p~ \} \leq \frac{2(n - 1)}{n - 2} \hspace{1 cm}\forall n
\geq 3.
\end{eqnarray}

Now, we state our main results as follows. In the first result, we
conclude blow up for solutions of problem \ref{1.1} and in the
second result, we present a lower bound for provided blow up time
$T$ under suitable conditions on components of our problem.

\begin{theorem}\label{th.1}
Assume that the conditions $A_{1}-A_{4}$ hold. Let $p$ and $m$ be
such that $m \geq 2$ and $p > \max\{2,m\}.$ Then, any weak solution
of \ref{1.1} with initial data satisfying the following conditions
blows up in finite time.
\[I(0) < I_{1} , \hspace{1 cm}
\|\nabla_{\mathbb{B}}u_{0}\|_{L^{\frac{n}{2}}(\mathbb{B})} >
(C_{*}~\sqrt[p]{C_{h}})^{-\frac{p}{p-2}} .\]
\end{theorem}

\begin{theorem}\label{th.2}
Suppose that the assumptions in Theorem \ref{th.1} hold. Further, we
deem that $u(x,t)$ is a weak solution of \ref{1.1}, which blows up
at a finite time $T.$ Then, $T$ admits a lower bound as follows
\[\int_{\mathcal{H}(0)}^{\infty} \frac{1}{\mathcal{C}_{3}S^{k} + S +
\mathcal{C}_{4}}ds \leq T.\]
\end{theorem}
Our paper organized as follows. Section 2 is devoted to introduce
and construct stretched manifold $\mathbb{B}$ associated to a given
manifold with conical singularities. Section 3 is concerned with
some subsidiary results about the energy functional corresponding to
the problem \ref{1.1}. Finally, section 4 is dedicated to proof of
our main results.

\section{Cone Sobolev Spaces}
In this section, we recall some definitions and notations from
Sobolev spaces on manifolds with conical singularities. We refer
enthusiastic readers to \cite{CLW1,S} and the references therein.

Let $B$ be a manifold with conical singularities $x_{1},...,x_{N}.$
First, for simplicity let us  consider the case $N =1$ and  set $x=
x_{1}$. If $X$ is $C^{\infty}$ closed compact manifold, the cone
$X^{\Delta} := \frac{\bar{\mathbb{R}}_{+}\times X}{\{0\}\times X}$
is an example of such an $B.$ In this case the conical singularity
is represented by $\{0\} \times X$ in the quotient space. In
general, $B$ is locally near $x$ modelled on such a cone. More
precisely, $B - \{x\}$ is smooth, and we have a singular chart
$$\chi : G \rightarrow X^{\Delta}$$ for some neighborhood $G$ of $x$
in $M$ and a smooth manifold $X= X(x)$ where $\chi(x)$ is equal to
the vertex of $X^{\Delta}$, $\varphi = \chi|_{G - \{x\}} : G- \{x\}
\rightarrow X^{\wedge}:= \mathbb{R}_{+}\times X$ is a diffeomorphism
\cite{CLW1}. More precisely,  a finite dimensional manifold $B$ with
conical singularities, is a topological space with a finite subset
$B_{0}= \{x_{1},...,x_{N}\}\subset B$ of conical singularities,
which has the following two properties:

$a)$ $B - B_{0}$ is a $C^{\infty}$ manifold.

$b)$ any $x\in B_{0}$ has an open neighborhood $G$ in $B$, such that
there is a homeomorphism $\chi :G \rightarrow X^{\Delta}$ for some
closed compact $C^{\infty}$ manifold $X = X(x)$ and $\varphi$
restricts a diffeomorphism $\varphi': G- \{x\} \rightarrow
X^{\wedge}.$

For such a manifold, let $n\geq 2 $ and $X \subset S^{n - 1}$ be a
bounded open set in the unit sphere of $\mathbb{R}_{x}^{n}.$ The set
$B:= \biggl\{ x\in \mathbb{R}^{n} - \{0\} ~~~; \hspace{0.3 cm}
\frac{x}{|x|}\in X\biggr\} \cup\{0\}$ is an infinite cone with the
base $X$ and the critical point $\{0\}.$ Using the polar
coordinates, one can get a description of $B - \{0\} $ in the form
$X^{\wedge} = \mathbb{R}_{+}\times X,$ which is called the open
stretched cone with the base $X,$ and $\{0\}\times X$ is the
boundary of $X^{\wedge}.$

Now, we assume that the manifold $B$ is paracompact and of dimension
$n.$  By this assumption we can define the stretched manifold
associated with $B$. Let $\mathbb{B}$ be a $C^{\infty}$ manifold
with compact $C^{\infty}$ boundary $\partial\mathbb{B}
\cong\bigcup\limits_{x\in B_{0}}X(x)$ for which there exists a
diffeomorphism $B - B_{0}\cong\mathbb{B}- \partial\mathbb{B} :=
int\mathbb{B},$ the restriction of which to $G_{1}- B_{0}\cong
U_{1}- \partial\mathbb{B}$ for an open neighborhood $G_{1}\subset B$
near the points of $B_{0}$ and a collar neighborhood $U_{1}\subset
\mathbb{B}$ with $U_{1}\cong\bigcup\limits_{x\in B_{0}}\{[0,1)
\times X(x)\}.$ The typical differential operators on a manifold
with conical singularities, called Fuchs type , are operators that
are in a neighborhood of $x_{1}=0$ of the following form $$A =
x_{1}^{-m}\sum_{k=0}^{m}a_{k}(x_{1}) (-x_{1} \partial_{x_{1}})^{k}$$
with $(x_{1},x)\in X^{\wedge}$ and $a_{k}(x_{1})\in
C^{\infty}(\bar{\mathbb{R}}_{+}, Diff^{m-k}(X))$ \cite{S}. The
differential $x_{1}\partial{x_{1}}$ in Fuchs type operators provokes
us to apply the Mellin transform $M: C_{0}^{\infty}(\mathbb{R}_{+})
\rightarrow \mathcal{A}(\mathbb{C}),$ for $u(x_{1})\in
C_{0}^{\infty}(\mathbb{R}_{+}),$ $z\in \mathbb{C},$ defined as
\begin{equation}\label{2.4}
M u(z):= \int_{0}^{+\infty}x_{1}^{z} u(x_{1})\frac{dx_{1}}{x_{1}},
\end{equation}
where $\mathcal{A}(\mathbb{C})$ denotes the space of entire
functions.

One can find further details on Fuchs type operators and all
implications and definitions of the cone Sobolev spaces in
\cite{CLW1,S}. We use the so-called weighted Melline transform
\[M_{\gamma}u:= M u_{|_{\Gamma_{\frac{1}{2} - \gamma}}} =
\int_{0}^{+\infty}x_{1}^{\frac{1}{2} - \gamma + i\tau} u(x_{1})
\frac{dx_{1}}{x_{1}},\] where $\Gamma_{\beta}:= \{z\in
\mathbb{C}~~~;\hspace{0.3 cm} Re z= \beta\}.$ The inverse weighted
Melline transform is defined as \[(M_{\gamma}^{-1} g)(x_{1}) =
\frac{1}{2i\pi}\int_{\Gamma_{\frac{1}{2}}}x_{1}^{-z} g(z) dz.\]

In order to define cone Sobolev spaces on the stretched manifolds,
at the first we introduce the weighted Sobolev spaces on
$\mathbb{R}^{n}$ and then by using of partition unity, we introduce
suitable weighted cone Sobolev space on the stretched manifold
$\mathbb{B}.$
\begin{definition}
For $(x_{1}, x')\in\mathbb{R}_{+} \times \mathbb{R}^{n - 1} =
\mathbb{R}^{n}_{+}$ we say that $u(x_{1},x')\in
L_{p}(\mathbb{R}^{n}_{+}, \frac{dx_{1}} {x_{1}}dx')$ if
\[\|u\|_{L_{p}}= \biggl(\int_{\mathbb{R}_{+}}\int_{\mathbb{R}^{n -
1}}x_{1}^{n}|u(x_{1},x')|^{p}
\frac{dx_{1}}{x_{1}}dx'\biggr)^{\frac{1} {p}} < \infty.\]
\end{definition}
The weighted $L_{p}-$spaces with weight data $\gamma \in \mathbb{R}$
is denoted by
$L_{p}^{\gamma}(\mathbb{R}_{+}^{n},\frac{dx_{1}}{x_{1}}dx'). $ In
fact, if $u(x_{1},x')\in
L_{p}^{\gamma}(\mathbb{R}_{+}^{n},\frac{dx_{1}}{x_{1}}dx'),$ then
$x_{1}^{-\gamma}u(x_{1},x')\in L_{p}
(\mathbb{R}_{+}^{n},\frac{dx_{1}}{x_{1}}dx'),$ and
\[\|u\|_{L^{p}}^{\gamma} =
\biggl(\int_{\mathbb{R}_{+}}\int_{\mathbb{R}^{n - 1}}
x_{1}^{n}|x_{1}^{
-\gamma}u(x_{1},x')|^{p}\frac{dx_{1}}{x_{1}}dx'\biggr)^{\frac{1}{p}}
< \infty.\]

Now, we can define the weighted $p-$Sobolev spaces for $1 \leq p <
\infty.$
\begin{definition}
For $m\in \mathbb{N},$ $\gamma\in \mathbb{R}$ and $1 \leq p <
\infty,$ the spaces
\begin{equation}\label{2.5}
\mathcal{H}^{m,\gamma}_{p}(\mathbb{R}_{+}^{n}):= \biggl\{u\in
\mathcal{D}'(\mathbb{R}^{n}_{+})~~~~;~~~~~ x_{1}^{\frac{n}{p} -
\gamma}( x_{1}\partial_{x_{1}})^{\alpha}\partial_{x'}^{\beta} u \in
L_{p}(\mathbb{R}^{n}_{+}, \frac{dx_{1}}{x_{1}}dx')\biggr\},
\end{equation}
for any $\alpha \in \mathbb{N},$ $\beta \in \mathbb{N}^{n - 1}$ and
$|\alpha| + |\beta| \leq m.$ In other words, if $u(x_{1},x)\in
\mathcal{H} ^{m , \gamma}_{p}(\mathbb{R}^{n}_{+}),$ then
$(x_{1}\partial_{x_{1}})^{\alpha}\partial_{x'}^{\beta} u \in
L_{p}^{\gamma}(\mathbb{R}^{n}_{+}, \frac{dx_{1}} {x_{1}}dx').$
\end{definition}
Hence, $\mathcal{H}^{m,\gamma}_{p}(\mathbb{R}^{n}_{+})$ is a Banach
space with norm \[\|u\|_{\mathcal{H}^{m ,
\gamma}_{p}(\mathbb{R}^{n}_{+})} = \sum_{|\alpha| + |\beta| \leq m}
\biggl(\int\int_{\mathbb{R}_{+}^{n}}
x_{1}^{n}|x_{1}^{-\gamma}(x_{1}\partial_{x_{1}})^{\alpha}\partial_{x'}^{\beta}
u(x_{1},x')|^{p} \frac{dx_{1}}{x_{1}}dx'\biggr)^{\frac{1}{p}}.\]

Let $X$ be a closed compact $C^{\infty}$ manifold, and $\mathcal{U}
= \{U_{1},..., U_{N}\}$ an open covering of $X$ by coordinate
neighborhoods. If we fix a subordinate partition of unity
$\{\varphi_{1},...,\varphi_{N}\}$ and charts $\chi_{j}: U_{j}
\rightarrow \mathbb{R}^{n -1},$ $j = 1, ...,N.$ Then we say that
$u\in \mathcal{H}^{m,\gamma}_{p}(X^{\wedge})$ if and only if $u\in
\mathcal{D}'(X^{\wedge})$ with the norm \[\|u\|_{\mathcal{H}^
{m,\gamma}_{p}(X^{\wedge})} = \biggl\{\sum_{j =1}^{N}\|(1 \times
\chi_{j}^{*})^{-1}\varphi_{j}
u\|^{p}_{\mathcal{H}^{m,\gamma}_{p}(\mathbb{R}_{+}^{n})}\biggr\}^{\frac{1}{p}}
< \infty,\] where $1 \times \chi_{j}^{*}:
C_{0}^{\infty}(\mathbb{R}_{+}\times \mathbb{R}^{n - 1}) \rightarrow
C_{0}^{\infty}(\mathbb{R}_{+} \times U_{j})$ is the pull-back
function with respect to $1 \times \chi_{j}: \mathbb{R}_{+}\times
U_{j} \rightarrow \mathbb{R}_{+}\times \mathbb{R}^{n - 1}.$ Denote
the $\mathcal{H}^{m,\gamma}_{p,0}(X^{\wedge})$ as subspace of
$\mathcal{H}^{m,\gamma}_{p}(X^{\wedge})$ which is defined as the
closure of $C_{0}^{\infty}(X^ {\wedge})$ with respect to the norm
$\| . \|_{\mathcal{H}^{m, \gamma}_{p}(X^{\wedge})}.$  Now, we have
the following definition
\begin{definition}
Let $\mathbb{B} = [0,1) \times X$ be the stretched manifold of the
manifold $B $ with conical singularity. Then the cone Sobolev space
$\mathcal{H}^{m,\gamma}_{p}(\mathbb{B}),$ for $m\in \mathbb{N} ,
\gamma\in \mathbb{R}$ and $p\in (1 , \infty),$ is defined by
$$\mathcal{H}^{m,\gamma}_{p}(\mathbb{B})= \{u \in W^{m,p}_{loc}(int\mathbb{B})~~~~~~~~;~~~~~~~~~~~~
\omega u\in \mathcal{H}^{m,\gamma}_{p}(X^{\wedge})\},$$ for any
cut-off function $\omega$ supported by a collar neighborhood of
$[0,1) \times \partial\mathbb{B}.$ Moreover, the subspace
$\mathcal{H}^{m,\gamma}_{p,0}(\mathbb{B})$ of
$\mathcal{H}^{m,\gamma}_{p}(\mathbb{B})$ is defined by
$$\mathcal{H}^{m,\gamma}_{p,0}(\mathbb{B}):= (\omega) \mathcal{H}^{m,\gamma}_{p,0}(X^{\wedge}) +
(1 - \omega) W^{m,p}_{0}(int\mathbb{B}),$$ where
$W^{m,p}_{0}(int\mathbb{B})$ denotes the closure of
$C_{0}^{\infty}(int \mathbb{B})$ in Sobolev space
$W^{m,p}(\tilde{X})$ when $\tilde{X}$ is closed compact
$C^{\infty}$manifold of dimension $n$ containing $\mathbb{B}$ as a
submanifold with boundary.
\end{definition}

\begin{definition}
Let  $\mathbb{B} = [0,1) \times X.$ We say that $u(x)\in
L^{\gamma}_{p}(\mathbb{B})$ with $1 < p <  \infty, \gamma\in
\mathbb{R},$ if $$\|u\|_{L^{\gamma}_{p}(\mathbb{B})} =
\int\limits_{\mathbb{B}}x_{1}^{n}|x_{1}^{-\gamma}
u(x)|^{p}(\frac{dx_{1}}{x_{1}}) dx' < \infty .$$
\end{definition}
For $\gamma = \frac{n}{p}$ and $\gamma = \frac{n}{q}$ such that
$\frac{1}{p} + \frac{1}{q} = 1,$  we have the following
H$\ddot{o}$lder's inequality
\begin{equation}\int\limits_{\mathbb{B}}|u(x) v(x)|
\frac{dx_{1}}{x_{1}} dx' \leq
\biggl(\int\limits_{\mathbb{B}}|u(x)|^{p}\frac{dx_{1}}{x_{1}}
dx'\biggr)^{\frac{1}{p}}\biggl(\int\limits_{\mathbb{B}}|v(x)|^{q}\frac{dx_{1}}{x_{1}}
dx'\biggr)^{\frac{1}{q}} .\end{equation} In the sequel, for
convenience we denote
\[( u, v )_{2} = \int\limits_{\mathbb{B}}
u(x)v(x)~\frac{dx_{1}}{x_{1}} dx'~~~and~~~
\|u\|_{L^{\frac{n}{p}}_{p}(\mathbb{B})} =
\int\limits_{\mathbb{B}}|u(x)|^{p}~\frac{dx_{1}}{x_{1}} dx'.\]

\section{Some auxiliary results}

In this section, we give some auxiliary results about the  problem
$\ref{1.1}$ and we get some properties of energy functional that we
will use to prove our main results. therefore, similar to the
classical case, we introduced  functional $I(t)= I(u(t),u_{t}(t)),$
on cone Sobolev space
$\mathcal{H}^{1,\frac{n}{2}}_{2,0}(\mathbb{B})$ as follows:
\begin{eqnarray}\label{2.1}
\nonumber I(t)&=& \frac{1}{2}
\|u_{t}(t)\|^{2}_{L^{\frac{n}{2}}(\mathbb{B})} + \frac{1}{2}
\biggl(1- \int_{0}^{t}
g(s)ds\biggr)\|\nabla_{\mathbb{B}}u(t)\|_{L^{\frac{n}{2}}(\mathbb{B})}^{2}
+ \frac{1}{2} (g~ o~ \nabla_{\mathbb{B}}u)(t) \\&-& \frac{1}{p}
\int_{\mathbb{B}} h(x)|u|^{p}\frac{dx_{1}}{x_{1}}dx',
\end{eqnarray}
where $$(g~ o~ v)(t) = \int_{0}^{t}g(t - \tau)\|v(t) -
v(\tau)\|^{2}_{L^{\frac{n}{2}}(\mathbb{B})} d\tau.$$

\begin{lemma}\label{l.1}
Suppose that conditions $A_{1}-A_{4}$ and \ref{2} hold.  Let $u$ be
a solution of the problem \ref{1.1}. Then the energy functional
$I(t)$ is non-increasing.
\end{lemma}
\begin{proof}
Let us multiply Eq. (1.1) by $u_{t}$ and integrating over
$\mathbb{B},$ then e obtain
\begin{eqnarray}\label{2.2}
\nonumber\langle u_{tt} , u_{t}\rangle &+& \langle
-\Delta_{\mathbb{B}}u , u_{t}\rangle + \langle\int_{0}^{t} g(t-\tau)
\Delta_{\mathbb{B}}u(\tau) d\tau , u_{t}\rangle + \langle
f(x)u_{t}|u_{t}|^{m-2}, u_{t}\rangle \\&=& \langle h(x)u|u|^{p-2},
u_{t}\rangle \Rightarrow
\end{eqnarray}
\begin{eqnarray}\label{2.3}
\frac{d}{dt}\biggl[ \frac{1}{2} \int_{\mathbb{B}}|u_{t}|^{2}
\frac{dx_{1}}{x_{1}}dx' &+& \frac{1}{2}\int_{\mathbb{B}}
|\nabla_{\mathbb{B}}u|^{2} \frac{dx_{1}}{x_{1}}dx' - \frac{1}{p}
\int_{\mathbb{B}}h(x)|u|^{p} \frac{dx_{1}}{x_{1}}dx' \biggr]
\\&-&\nonumber \int_{0}^{t}
g(t-\tau)\int_{\mathbb{B}}\nabla_{\mathbb{B}}u_{t}(t).\nabla_{\mathbb{B}}u(\tau)
\frac{dx_{1}}{x_{1}}dx' d\tau =
-\int_{\mathbb{B}}f(x)|u_{t}|^{m}\frac{dx_{1}}{x_{1}}dx'.
\end{eqnarray}
On the other hand,
\begin{eqnarray}\label{2.4}
\nonumber\int_{0}^{t}g(t-\tau)\int_{\mathbb{B}}\nabla_{\mathbb{B}}u_{t}(t)&.&\nabla_{\mathbb{B}}u(\tau)\frac{dx_{1}}{x_{1}}dx'd\tau
=
\int_{0}^{t}g(t-\tau)\int_{\mathbb{B}}\nabla_{\mathbb{B}}u_{t}(t).(\nabla_{\mathbb{B}}u(\tau)
- \nabla_{\mathbb{B}}u(t))\frac{dx_{1}}{x_{1}}dx'd\tau
\\&+&\nonumber
\int_{0}^{t}g(t-\tau)\int_{\mathbb{B}}\nabla_{\mathbb{B}}u_{t}(t).\nabla_{\mathbb{B}}u(t)\frac{dx_{1}}{x_{1}}dx'd\tau
\\&=&\nonumber
-\frac{1}{2}\int_{0}^{t}g(t-\tau)\frac{d}{dt}\int_{\mathbb{B}}|\nabla_{\mathbb{B}}u(\tau)
- \nabla_{\mathbb{B}}u(t)|^{2}\frac{dx_{1}}{x_{1}}dx'd\tau
\\&+&\nonumber
\int_{0}^{t}g(\tau)\biggl(\frac{d}{dt}\frac{1}{2}\int_{\mathbb{B}}|\nabla_{\mathbb{B}}u(t)|^{2}\frac{dx_{1}}{x_{1}}dx'\biggr)d\tau
\\&=&\nonumber -
\frac{1}{2}\frac{d}{dt}\biggl[\int_{0}^{t}g(t-\tau)\int_{\mathbb{B}}|\nabla_{\mathbb{B}}u(\tau)
- \nabla_{\mathbb{B}}u(t)|^{2}\frac{dx_{1}}{x_{1}}dx'd\tau\biggr]
\\&+&\nonumber
\frac{1}{2}\frac{d}{dt}\biggl[\int_{0}^{t}g(\tau)\int_{\mathbb{B}}|\nabla_{\mathbb{B}}u(t)|^{2}\frac{dx_{1}}{x_{1}}dx'd\tau\biggr]\\
&+&\nonumber
\frac{1}{2}\int_{0}^{t}g'(t-\tau)\int_{\mathbb{B}}|\nabla_{\mathbb{B}}u(\tau)
- \nabla_{\mathbb{B}}u(t)|^{2} \frac{dx_{1}}{x_{1}}dx'd\tau
\\&-&\frac{1}{2}g(t)\int_{\mathbb{B}}|\nabla_{\mathbb{B}}u(t)|^{2}\frac{dx_{1}}{x_{1}}dx'd\tau.
\end{eqnarray}
Now, we situate estimations \ref{2.4} in the relation \ref{2.3}, we
get
\begin{eqnarray}\label{2.5}
\nonumber\frac{d}{dt}\biggl[\frac{1}{2}\int_{\mathbb{B}}|u_{t}|^{2}\frac{dx_{1}}{x_{1}}dx'
&+&
\frac{1}{2}\int_{\mathbb{B}}|\nabla_{\mathbb{B}}u|^{2}\frac{dx_{1}}{x_{1}}dx'
-
\frac{1}{p}\int_{\mathbb{B}}h(x)|u|^{p}\frac{dx_{1}}{x_{1}}dx'\biggr]\\&+&\nonumber
\frac{1}{2}\frac{d}{dt}\biggl[\int_{0}^{t}g(t-\tau)\int_{\mathbb{B}}|\nabla_{\mathbb{B}}u(\tau)
- \nabla_{\mathbb{B}}u(t)|^{2}\frac{dx_{1}}{x_{1}}dx'd\tau\biggr]
\\&-&
\nonumber\frac{1}{2}\frac{d}{dt}\biggl[\int_{0}^{t}g(\tau)\|\nabla_{\mathbb{B}}u(t)\|^{2}d\tau\biggr]\\&=&\nonumber
-\int_{\mathbb{B}}f(x)|u_{t}|^{m}\frac{dx_{1}}{x_{1}}dx' +
\frac{1}{2}
\int_{0}^{t}g'(t-\tau)\int_{\mathbb{B}}|\nabla_{\mathbb{B}}u(\tau) -
\nabla_{\mathbb{B}}u(t)|^{2}\frac{dx_{1}}{x_{1}}dx'd\tau\\&-&\frac{1}{2}g(t)\|\nabla_{\mathbb{B}}u\|_{L^{\frac{n}{2}}(\mathbb{B})}^{2}
\leq 0.
\end{eqnarray}
Therefore, $I'(t)\leq 0.$ It follows that $I(t)$ is non-increasing.
\end{proof}

\begin{lemma}\label{l.2}
Let $u$ be a weak solution of \ref{1.1} such that conditions
$A_{1}-A_{4}$ and relation \ref{2} hold. Moreover, assume that
\begin{eqnarray}\label{l.2.1}
I(0) < I_{1} , \hspace{1 cm}
\|\nabla_{\mathbb{B}}u_{0}\|_{L^{\frac{n}{2}}(\mathbb{B})} >
(C_{*}~\sqrt[p]{C_{h}})^{-\frac{p}{p-2}} .
\end{eqnarray}
Then, there exists a constant $\beta >
(C_{*}~\sqrt[p]{C_{h}})^{-\frac{p}{p-2}}$ such that for any $t \in
[0, T)$
\begin{eqnarray}\label{l.2.2}
\biggl[\frac{1}{2}\biggl(1 -
\int_{0}^{t}g(s)ds\biggr)\|\nabla_{\mathbb{B}}u(t)\|_{L^{\frac{n}{2}}(\mathbb{B})}^{2}
+ \frac{1}{2} (g~ o~ \nabla_{\mathbb{B}}u)(t)\biggr]^{\frac{1}{2}}
\geq \beta
\end{eqnarray}
and for every $t \in [0, T)$
\begin{eqnarray}\label{l.2.3}
\|u\|_{\mathcal{H}^{1,\frac{n}{2}}_{2,0}(\mathbb{B})} \geq
\frac{\beta}{\sqrt{l}}.
\end{eqnarray}
\end{lemma}
\begin{proof}
According to definition of the functional $I,$
\begin{eqnarray}\label{2.6}
\nonumber I(t)&=& \frac{1}{2}
\|u_{t}(t)\|^{2}_{L^{\frac{n}{2}}(\mathbb{B})} + \frac{1}{2}
\biggl(1- \int_{0}^{t}
g(s)ds\biggr)\|\nabla_{\mathbb{B}}u(t)\|_{L^{\frac{n}{2}}(\mathbb{B})}^{2}
+ \frac{1}{2} (g~ o~ \nabla_{\mathbb{B}}u)(t) \\&-&\nonumber
\frac{1}{p} \int_{\mathbb{B}} h(x)|u|^{p}\frac{dx_{1}}{x_{1}}dx'\geq
\frac{1}{2}\biggl(1- \int_{0}^{t}
g(s)ds\biggr)\|\nabla_{\mathbb{B}}u(t)\|_{L^{\frac{n}{2}}(\mathbb{B})}^{2}
+ \frac{1}{2} (g~ o~ \nabla_{\mathbb{B}}u)(t) \\&-& \frac{1}{p}
\int_{\mathbb{B}} h(x)|u|^{p}\frac{dx_{1}}{x_{1}}dx'.
\end{eqnarray}
On the other hand, by relation \ref{1.3} and H\"{o}lder inequality
we can obtain the following to estimations \ref{2.6}
\begin{eqnarray}\label{2.7}
\nonumber\frac{1}{2}\biggl(1 &-& \int_{0}^{t}
g(s)ds\biggr)\|\nabla_{\mathbb{B}}u(t)\|_{L^{\frac{n}{2}}(\mathbb{B})}^{2}
+ \frac{1}{2} (g~ o~ \nabla_{\mathbb{B}}u)(t) \\&-&\nonumber
\frac{C_{*}^{p}l^{p}
C_{h}}{p}\|u\|^{p}_{\mathcal{H}^{1,\frac{n}{2}}_{2}(\mathbb{B})}
\geq \frac{1}{2}\biggl(1- \int_{0}^{t}
g(s)ds\biggr)\|\nabla_{\mathbb{B}}u(t)\|_{L^{\frac{n}{2}}(\mathbb{B})}^{2}
+ \frac{1}{2} (g~ o~ \nabla_{\mathbb{B}}u)(t) \\&-&
\frac{C_{*}^{p}C_{h}}{p} \biggl[\biggl(1- \int_{0}^{t}
g(s)ds\biggr)\|\nabla_{\mathbb{B}}u(t)\|_{L^{\frac{n}{2}}(\mathbb{B})}^{2}
+ \frac{1}{2} (g~ o~ \nabla_{\mathbb{B}}u)(t)\biggr]^{\frac{p}{2}}.
\end{eqnarray}
If in relation \ref{2.7}, we take $\biggl[\frac{1}{2}\biggl(1-
\int_{0}^{t}
g(s)ds\biggr)\|\nabla_{\mathbb{B}}u(t)\|_{L^{\frac{n}{2}}(\mathbb{B})}^{2}
+ \frac{1}{2} (g~ o~ \nabla_{\mathbb{B}}u)(t)\biggr]^{\frac{1}{2}} =
\theta$ then we have a function
\begin{eqnarray}\label{2.8}
k(\theta)= \frac{1}{2}\theta^{2}  -
\frac{C_{*}^{p}C_{h}}{p}\theta^{p}.
\end{eqnarray}
By simple calculations, one can get
$$\max\limits_{\theta}k(\theta) = k(\alpha),$$ where \begin{equation}\label{2.9}\alpha:=
(\frac{1}{C_{*}^{p}C_{h}})^{\frac{1}{p-2}}.\end{equation} Then for
any $\theta$
\begin{equation}\label{2.10}k(\alpha) =
\biggl(\frac{1}{C_{*}^{p}C_{h}}\biggr)^{\frac{2}{p-2}}
\biggl[\frac{1}{2} - \frac{C_{*}^{p}C_{h}}{p}\biggr] \geq
k(\theta).\end{equation} Hence, for any $0 < \theta < \alpha,$
$k(\theta)$ is increasing and for any $\theta > \alpha ,$
$k(\theta)$ is decreasing. It follows that as $\theta\rightarrow
\infty,$ $k(\theta)\rightarrow -\infty$ and

\begin{equation}\label{2.11}
k(\alpha)= \biggl(\frac{1}{C_{*}^{p}C_{h}}\biggr)^{\frac{2}{p-2}}
\biggl[\frac{1}{2} - \frac{C_{*}^{p}C_{h}}{p}\biggr]= I_{1}
\end{equation}
where $\alpha$ is given by \ref{2.9}. Using of the assumption $I(0)
< I_{1}$ there exists a constant $\beta > \alpha$ for which
$k(\beta) = I(0).$ Now, we set $\alpha_{0} =
\|\nabla_{\mathbb{B}}u_{0}\|_{L^{\frac{n}{2}}(\mathbb{B})},$ then by
\ref{2.8}
\begin{eqnarray}\label{2.12}
\nonumber k(\alpha_{0})&=& \frac{1}{2}\alpha_{0}^{2} -
\frac{C_{*}^{p}C_{h}}{p}\alpha_{0}^{p} = \frac{1}{2}
\|\nabla_{\mathbb{B}}u_{0}\|^{2}_{L^{\frac{n}{2}}(\mathbb{B})} -
\frac{C_{*}^{p}C_{h}}{p}\|\nabla_{\mathbb{B}}u_{0}\|^{p}_{L^{\frac{n}{2}}(\mathbb{B})}\\&\leq&
I(0) = k(\beta).
\end{eqnarray}
It follows that $\alpha_{0} > \beta.$

Now, we use contradiction method to show that relation \ref{l.2.2}
is satisfied. Suppose that there exist some $t_{0} >0$ such that
\begin{eqnarray}\label{2.13}
\biggl[(1 -
\int_{0}^{t_{0}}g(s)ds)\|\nabla_{\mathbb{B}}u\|^{2}_{L^{\frac{n}{2}}(\mathbb{B})}
+ (g~o~\nabla_{\mathbb{B}}u)(t_{0})\biggr]^{\frac{1}{2}} < \beta.
\end{eqnarray}
Since \[\biggl(1 -
\int_{0}^{t_{0}}g(s)ds\biggr)\|\nabla_{\mathbb{B}}u\|^{2}_{L^{\frac{n}{2}}(\mathbb{B})}
+ (g~o~\nabla_{\mathbb{B}}u)(t_{0})\] is continuous, so one can
choose $t_{0}$ for which
\begin{eqnarray}\label{2.14}
\biggl[(1 -
\int_{0}^{t_{0}}g(s)ds)\|\nabla_{\mathbb{B}}u\|^{2}_{L^{\frac{n}{2}}(\mathbb{B})}
+ (g~o~\nabla_{\mathbb{B}}u)(t_{0})\biggr]^{\frac{1}{2}} > \alpha.
\end{eqnarray}
Using of relations \ref{2.8} and \ref{2.10}, one can obtain
\[I(t_{0}) \geq k(\biggl( \biggl[(1 -
\int_{0}^{t_{0}}g(s)ds)\|\nabla_{\mathbb{B}}u\|^{2}_{L^{\frac{n}{2}}(\mathbb{B})}
+ (g~o~\nabla_{\mathbb{B}}u)(t_{0})\biggr]^{\frac{1}{2}} \biggr) >
k(\beta) = I(0).\] But, this is contradiction beacuse $I(t) \leq
I(0)$ for any $t\in[0, T).$ Therefore, the relation \ref{l.2.2} is
satisfied.

Now, we prove the relation \ref{l.2.3}. According to the definition
of the energy functional $I(t)$ and H\"{o}lder inequality we have,
\begin{eqnarray}\label{2.15}
\nonumber\frac{1}{2}\biggl [(1 &-&
\int_{0}^{t_{0}}g(s)ds)\|\nabla_{\mathbb{B}}u\|^{2}_{L^{\frac{n}{2}}(\mathbb{B})}
+ (g~o~\nabla_{\mathbb{B}}u)(t_{0})\biggr]^{\frac{1}{2}} \leq I(0) +
\frac{1}{p}\int_{\mathbb{B}}h(x)|u|^{p}\frac{dx_{1}}{x_{1}}dx'\\&\leq&
I(0)  ~+~
\frac{C_{h}C_{*}^{p}l^{\frac{p}{2}}}{p}\|u\|^{p}_{\mathcal{H}^{1,\frac{n}{2}}_{2}(\mathbb{B})}.
\end{eqnarray}
Therefore,
\begin{eqnarray}\label{2.16}
\nonumber\frac{C_{h}C_{*}^{p}l^{\frac{p}{2}}}{p}\|u\|^{p}_{\mathcal{H}^{1,\frac{n}{2}}_{2,0}(\mathbb{B})}
&\geq& \frac{1}{2}\biggl [(1 -
\int_{0}^{t_{0}}g(s)ds)\|\nabla_{\mathbb{B}}u\|^{2}_{L^{\frac{n}{2}}(\mathbb{B})}
+ (g~o~\nabla_{\mathbb{B}}u)(t_{0})\biggr]^{\frac{1}{2}} \\&-& I(0)
\geq \frac{1}{2}\beta^{2} - I(0) \geq \frac{1}{2}\beta^{2} -
k(\beta) =  \frac{C_{*}^{p}C_{h}}{p}\beta^{p}.
\end{eqnarray}
Hence,
\begin{equation}\label{2.17}\|u\|^{p}_{\mathcal{H}^{1,\frac{n}{2}}_{2,0}(\mathbb{B})} \geq
\frac{\beta^{p}}{l^{\frac{p}{2}}}.\end{equation}
\end{proof}

\begin{lemma}\label{l.3}
Suppose that \ref{2} holds. Then there exists a positive constant $M
$ such that
\begin{equation}\label{2.18}
\|u\|^{s}_{L^{\frac{n}{p}}(\mathbb{B})} \leq
M\biggl(\|\nabla_{\mathbb{B}}u\|^{2}_{L^{\frac{n}{2}}(\mathbb{B})} +
\|u\|^{p}_{L^{\frac{n}{p}}(\mathbb{B})}\biggr),
\end{equation}
for any $u\in \mathcal{H}^{1,\frac{n}{2}}_{2,0}(\mathbb{B})$ and $s
\in [2, p].$
\end{lemma}
\begin{proof}
We consider two cases. First, if
$\|u\|_{L^{\frac{n}{p}}(\mathbb{B})} \leq 1,$ by Sobolove embedding
Theorem and Poincar\'{e}'s inequality we obatian
\begin{eqnarray*}\|u\|^{s}_{L^{\frac{n}{p}}(\mathbb{B})} \leq
\|u\|^{2}_{L^{\frac{n}{p}}(\mathbb{B})} &\leq& C_{Poin}
\|\nabla_{\mathbb{B}}u\|^{2}_{L^{\frac{n}{p}}(\mathbb{B})} +
 \|u\|^{2}_{L^{\frac{n}{p}}(\mathbb{B})} \\&\leq &
C_{*}l^{\frac{1}{2}}C_{Poin}
\|\nabla_{\mathbb{B}}u\|^{2}_{L^{\frac{n}{2}}(\mathbb{B})} +
 \|u\|^{2}_{L^{\frac{n}{p}}(\mathbb{B})}\\&\leq& C_{*}l^{\frac{1}{2}}C_{Poin}
\|\nabla_{\mathbb{B}}u\|^{2}_{L^{\frac{n}{2}}(\mathbb{B})} +
 \|u\|^{p}_{L^{\frac{n}{p}}(\mathbb{B})}\\&\leq& M \biggl(\|\nabla_{\mathbb{B}}u\|^{2}_{L^{\frac{n}{2}}(\mathbb{B})} +
 \|u\|^{p}_{L^{\frac{n}{p}}(\mathbb{B})}\biggr),
\end{eqnarray*}
where $M > \max\{C_{*}l^{\frac{1}{2}}C_{Poin} ~, 1\}.$ Now, we
assume that $u\in \mathcal{H}^{1,\frac{1}{2}}_{2,0}(\mathbb{B})$
such that $\|u\|_{L^{\frac{n}{p}}(\mathbb{B})} \geq 1,$ then
\[\|u\|^{s}_{L^{\frac{n}{p}}(\mathbb{B})} \leq \|u\|^{p}_{L^{\frac{n}{p}}(\mathbb{B})} \leq
M \biggl(\|\nabla_{\mathbb{B}}u\|^{2}_{L^{\frac{n}{2}}(\mathbb{B})}
+
 \|u\|^{p}_{L^{\frac{n}{p}}(\mathbb{B})}\biggr).\]
\end{proof}
We set $\mathcal{I}(t) = I_{1} - I(t)$ and use positive constant $M$
which depends only on $p$ and $l.$ Here, we prove a result about
function $\mathcal{I}(t)$ by using of Lemma \ref{l.3} and \ref{l.2}.

\begin{proposition}\label{p.1}
Suppose that $u$ is a weak solution of \ref{1.1} and \ref{2} holds.
Then, there exists a positive constant $\delta=\delta(l,p)$ such
that for any $s \in [2 , p]$
\begin{eqnarray}\label{p.1.1}
\|u\|^{s}_{L^{\frac{n}{p}}(\mathbb{B})} \leq
\Lambda\biggl[-2\mathcal{I}(t) -
\|u_{t}\|^{2}_{L^{\frac{n}{2}}(\mathbb{B})} -
(g~o~\nabla_{\mathbb{B}}u)(t)  +
\|u\|^{p}_{L^{\frac{n}{p}}(\mathbb{B})}\biggr]
\end{eqnarray}
where, $\Lambda = \Lambda(M,\delta, C_{h})$ is a positive constant.
\end{proposition}
\begin{proof}
According to assumption $A_{1}$ and relation \ref{2.1} one can get

\begin{eqnarray}\label{2.19}
\nonumber\frac{1}{2}(1 -
l)\|\nabla_{\mathbb{B}}u\|_{L^{\frac{n}{2}}(\mathbb{B})}^{2} &\leq&
\frac{1}{2}(1 - \int_{0}^{t}
g(s)ds)\|\nabla_{\mathbb{B}}u\|_{L^{\frac{n}{2}}(\mathbb{B})}^{2}
\nonumber \\ &\leq& I(t) -
\frac{1}{2}\|u_{t}\|_{L^{\frac{n}{2}}(\mathbb{B})}^{2} -
\frac{1}{2}(g~o~\nabla_{\mathbb{B}}u)(t) + \frac{1}{p}
\int_{\mathbb{B}}h(x)|u|^{p}\frac{dx_{1}}{x_{1}}dx' \nonumber \\
&\leq& I_{1} - \mathcal{I}(t) - \frac{1}{2}
\|u_{t}\|^{2}_{L^{\frac{n}{2}}(\mathbb{B})} - \frac{1}{2}
(g~o~\nabla_{\mathbb{B}} u)(t) \nonumber\\&+& \frac{1}{p}
\int_{\mathbb{B}}h(x) |u|^{p}\frac{dx_{1}}{x_{1}} dx' \nonumber \\
&\leq& I_{1} - \mathcal{I}(t) - \frac{1}{2}
\|u_{t}\|^{2}_{L^{\frac{n}{2}}(\mathbb{B})} - \frac{1}{2}
(g~o~\nabla_{\mathbb{B}} u)(t) +
\frac{C_{h}}{p}\|u\|^{p}_{L^{\frac{n}{p}}(\mathbb{B})}.
\end{eqnarray}
On the other hand, by relations \ref{2.11} and \ref{2.17}, we obtain
the following estiamtion for $I_{1},$
\begin{eqnarray}\label{2.20}
\nonumber I_{1} &=& \biggl(\frac{1}{C_{*}^{p}
C_{h}}\biggr)^{\frac{2}{p-2}}\biggl[\frac{1}{2} - \frac{C_{*}^{p}
C_{h}}{p}\biggr] \leq \beta^{2} \biggl[\frac{1}{2} - \frac{C_{*}^{p}
C_{h}}{p}\biggr] \nonumber\\&\leq& \frac{1}{l} \biggl[\frac{1}{2} -
\frac{C_{*}^{p}
C_{h}}{p}\biggr]\|u\|^{2}_{\mathcal{H}^{1,\frac{n}{2}}_{2}(\mathbb{B})}
\nonumber\\&=& \frac{1}{l}\biggl[\frac{1}{2} - \frac{C_{*}^{p}
C_{h}}{p}\biggr] \|\nabla_{\mathbb{B}}
u\|^{2}_{L^{\frac{n}{2}}(\mathbb{B})}.
\end{eqnarray}
From \ref{2.20}, we get
\begin{eqnarray}\label{2.21}
\nonumber\frac{1}{2}(1 -
l)\|\nabla_{\mathbb{B}}u\|_{L^{\frac{n}{2}}(\mathbb{B})}^{2} &\leq&
\frac{1}{l}\biggl[\frac{1}{2} - \frac{C_{*}^{p} C_{h}}{p}\biggr]
\|\nabla_{\mathbb{B}} u\|^{2}_{L^{\frac{n}{2}}(\mathbb{B})} \\&-&
\mathcal{I}(t) - \frac{1}{2}
\|u_{t}\|^{2}_{L^{\frac{n}{2}}(\mathbb{B})} -
\frac{1}{2}(g~o~\nabla_{\mathbb{B}} u)(t) +
\frac{C_{h}}{p}\|u\|^{p}_{L^{\frac{n}{p}}(\mathbb{B})}.
\end{eqnarray}
Hence, by relation \ref{2.21}, one can get
\begin{eqnarray*}
\biggl[\frac{2C_{*}^{p}C_{h} - p(l^{2} -l +
1)}{2lp}\biggr]\|\nabla_{\mathbb{B}}
u\|^{2}_{L^{\frac{n}{2}}(\mathbb{B})} &\leq& -\mathcal{I}(t) -
\frac{1}{2}\|u_{t}\|^{2}_{L^{\frac{n}{2}}(\mathbb{B})} - \frac{1}{2}
(g~o~\nabla_{\mathbb{B}} u)(t) + \frac{C_{h}}{p}
\|u\|^{p}_{L^{\frac{n}{p}}(\mathbb{B})}.
\end{eqnarray*}
Then,
\begin{eqnarray}\label{2.22}
\|\nabla_{\mathbb{B}} u\|^{2}_{L^{\frac{n}{2}}(\mathbb{B})} &\leq&
\frac{1}{\delta(l,p)} \biggl[- \mathcal{I}(t) -
\frac{1}{2}\|u_{t}\|^{2}_{L^{\frac{n}{2}}(\mathbb{B})} - \frac{1}{2}
(g~o~\nabla_{\mathbb{B}} u)(t) + \frac{C_{h}}{p}
\|u\|^{p}_{L^{\frac{n}{p}}(\mathbb{B})}\biggr],
\end{eqnarray}
where $\delta=\delta(l,p) = \frac{2C_{*}^{p}C_{h} - p(l^{2} -l +
1)}{2lp}.$

Now, we apply Lemma \ref{l.3}, then
\begin{eqnarray}\label{2.23}
\nonumber\|u\|^{s}_{L^{\frac{n}{p}}(\mathbb{B})} &\leq&
\frac{2}{\delta M}\biggl[ - 2\mathcal{I}(t)  -
\|u_{t}\|^{2}_{L^{\frac{n}{2}}(\mathbb{B})} -
(g~o~\nabla_{\mathbb{B}} u)(t)\biggr] \nonumber\\&+& \biggl(M +
\frac{C_{h}}{\delta M p}\biggr)
\|u\|^{p}_{L^{\frac{n}{p}}(\mathbb{B})}.
\end{eqnarray}
It follows from \ref{2.23} that
\begin{eqnarray}\label{2.24}
\|u\|^{s}_{L^{\frac{n}{p}}(\mathbb{B})} &\leq&
\Lambda\biggl[-2\mathcal{I}(t) -
\|u_{t}\|^{2}_{L^{\frac{n}{2}}(\mathbb{B})} -
(g~o~\nabla_{\mathbb{B}} u)(t) +
\|u\|^{p}_{L^{\frac{n}{p}}(\mathbb{B})}\biggr],
\end{eqnarray}
where
\begin{equation}\label{2.25}
\Lambda = \Lambda(M,\delta, C_{h}) = \max\{\frac{2}{\delta M} , M +
\frac{C_{h}}{\delta pM}\}
\end{equation}

\section{proof of Main results}

{\bf Proof of Theorem \ref{th.1}}

According to relation \ref{2.1} and definition of $\mathcal{I}(t),$
we get that
\begin{eqnarray}\label{2.26}
\nonumber 0 < \mathcal{I}(0) &=& I_{1} - I(0) \leq I_{1} - I(t) =
\mathcal{I}(t) \nonumber\\&=& I_{1} -
\biggl[\frac{1}{2}\|u_{t}\|^{2}_{L^{\frac{n}{2}}(\mathbb{B})} +
\frac{1}{2} (1 - \int_{0}^{t}g(s)ds)\|\nabla_{\mathbb{B}}
u\|^{2}_{L^{\frac{n}{2}}(\mathbb{B})} \nonumber\\&+&
\frac{1}{2}(g~o~\nabla_{\mathbb{B}} u)(t) -
\frac{1}{p}\int_{\mathbb{B}}h(x)|u|^{p}\frac{dx_{1}}{x_{1}}dx'\biggr]
\leq I_{1} -
\frac{1}{2}\biggl[\|u_{t}\|_{L^{\frac{n}{2}}(\mathbb{B})}^{2}
\nonumber\\&+& \frac{1}{2} (1 -
\int_{0}^{t}g(s)ds)\|\nabla_{\mathbb{B}}
u\|^{2}_{L^{\frac{n}{2}}(\mathbb{B})} +
\frac{1}{2}(g~o~\nabla_{\mathbb{B}} u)(t) \biggr] +
\frac{C_{h}}{p}\|u\|^{p}_{L^{\frac{n}{p}}(\mathbb{B})}.
\end{eqnarray}
From Lemma \ref{l.2}, for any $t\in [0, \infty),$
\begin{eqnarray}\label{2.27}
\nonumber I_{1} &-&
\frac{1}{2}\biggl[\|u_{t}\|_{L^{\frac{n}{2}}(\mathbb{B})}^{2} +
\frac{1}{2} (1 - \int_{0}^{t}g(s)ds)\|\nabla_{\mathbb{B}}
u\|^{2}_{L^{\frac{n}{2}}(\mathbb{B})} +
\frac{1}{2}(g~o~\nabla_{\mathbb{B}} u)(t) \biggr] \nonumber\\&<&
\alpha^{2}\biggl(\frac{1}{2} - \frac{C_{*}^{p}C_{h}}{p}\biggr) -
\frac{1}{2}\beta^{2} = - \frac{\beta^{2}C_{*}^{p}C_{h}}{p} < 0.
\end{eqnarray}

Therefore,
\begin{eqnarray}\label{2.28}
0 < \mathcal{I}(0) \leq \mathcal{I}(t) \leq
\frac{C_{h}}{p}\|u\|^{p}_{L^{\frac{n}{p}}(\mathbb{B})} \hspace{1
cm}\forall t\geq 0.
\end{eqnarray}
Now, we consider $$0 < \epsilon \leq
\frac{m(1-k)C_{*}C_{Poin}^{m}}{N(m-1)},$$ and define
\begin{equation}\label{2.29}
\mathcal{F}(t) := \mathcal{I}^{1-k}(t) + \epsilon\int_{\mathbb{B}} u
u_{t}\frac{dx_{1}}{x_{1}}dx',
\end{equation}
such that \[0 < k\leq \min\biggl \{\frac{p-2}{2p} ,
\frac{p-m}{p(m-1)}\biggr\}.\]Moreover, $$\mathcal{F}(0)=
\mathcal{I}^{1-k}(0) +
\epsilon\int_{\mathbb{B}}u_{0}(x)~u_{1}(x)\frac{dx_{1}}{x_{1}}dx'
> 0 .$$ We take a derivative from \ref{2.29} and use \ref{1.1}, then
\begin{eqnarray}\label{2.30}
\nonumber\mathcal{F}^{'}(t) &=& (1-k)\mathcal{I}^{-k}(t)
\mathcal{I}^{'}(t) + \epsilon\int_{\mathbb{B}}[ u_{t}^{2} + u u_{tt}
]\frac{dx_{1}}{x_{1}}dx'\nonumber\\&=&
(1-k)\mathcal{I}^{-k}(t)\biggl[\int_{\mathbb{B}}f(x)|u_{t}|^{m}\frac{dx_{1}}{x_{1}}dx'
-  \frac{1}{2}\int_{0}^{t}
g'(t-\tau)\int_{\mathbb{B}}|\nabla_{\mathbb{B}}u(t) -
\nabla_{\mathbb{B}}u(\tau)|^{2} \frac{dx_{1}}{x_{1}}dx' d\tau
\nonumber\\&+& \frac{1}{2}g(t)\|\nabla_{\mathbb{B}}
u\|_{L^{\frac{n}{2}}(\mathbb{B})}^{2} \biggr]\nonumber\\&+&
\epsilon\int_{\mathbb{B}}\biggl[ u_{t}^{2} + h(x)|u|^{p} -
f(x)u_{t}|u_{t}|^{m-2}u ~-~ \int_{0}^{t}g(t-\tau)
\Delta_{\mathbb{B}} u(\tau) u d\tau + (\Delta_{\mathbb{B}}u)u\biggr]
\frac{dx_{1}}{x_{1}}dx'\nonumber\\&=&(1-k)\mathcal{I}^{-k}(t)\biggl[\int_{\mathbb{B}}f(x)|u_{t}|^{m}\frac{dx_{1}}{x_{1}}dx'
-  \frac{1}{2}\int_{0}^{t}
g'(t-\tau)\int_{\mathbb{B}}|\nabla_{\mathbb{B}}u(t) -
\nabla_{\mathbb{B}}u(\tau)|^{2} \frac{dx_{1}}{x_{1}}dx' d\tau
\nonumber\\&+&\frac{1}{2}g(t)\|\nabla_{\mathbb{B}}
u\|_{L^{\frac{n}{2}}(\mathbb{B})}^{2} \biggr] +
\epsilon\int_{\mathbb{B}}[~ u_{t}^{2}  - |\nabla_{\mathbb{B}}
u|^{2}~] \frac{dx_{1}}{x_{1}}dx' \nonumber\\&+&
\epsilon\int_{0}^{t}g(t-\tau)\int_{\mathbb{B}}\nabla_{\mathbb{B}}
u(t) . \nabla_{\mathbb{B}} u(\tau) \frac{dx_{1}}{x_{1}}dxd\tau +
\epsilon\int_{\mathbb{B}}h(x)|u|^{p} \frac{dx_{1}}{x_{1}}dx'
\nonumber\\&-& \epsilon\int_{\mathbb{B}}f(x)|u_{t}|^{m-2} u_{t} u
\frac{dx_{1}}{x_{1}}dx'.
\end{eqnarray}
Therefore,
\begin{eqnarray}\label{2.31}
\mathcal{F}^{'}(t) &\geq&
(1-k)\mathcal{I}^{-k}(t)\int_{\mathbb{B}}f(x)|u_{t}|^{m}
\frac{dx_{1}}{x_{1}}dx' + \epsilon\int_{\mathbb{B}}[~ u_{t}^{2}  -
|\nabla_{\mathbb{B}} u|^{2}~] \frac{dx_{1}}{x_{1}}dx'
\nonumber\\&+&\epsilon\int_{\mathbb{B}}h(x)|u|^{p}\frac{dx_{1}}{x_{1}}dx'
~-~ \epsilon\int_{\mathbb{B}}f(x)|u_{t}|^{m-2}u_{t} u
\frac{dx_{1}}{x_{1}}dx' +
\epsilon\int_{0}^{t}g(t-\tau)\|\nabla_{\mathbb{B}}u\|^{2}_{L^{\frac{n}{2}}(\mathbb{B})}d\tau
\nonumber\\&+&\epsilon\int_{0}^{t}g(t-\tau)\int_{\mathbb{B}}\nabla_{\mathbb{B}}
u(t) .[~\nabla_{\mathbb{B}} u(\tau) - \nabla_{\mathbb{B}}
u(t)~]\frac{dx_{1}}{x_{1}}dx'd\tau.
\end{eqnarray}
Now, we apply the Schwartz inequality, then we obatin
\begin{eqnarray}\label{2.32}
\mathcal{F}^{'}(t) &\geq&
(1-k)\mathcal{I}^{-k}(t)\int_{\mathbb{B}}f(x)|u_{t}|^{m}
\frac{dx_{1}}{x_{1}}dx' + \epsilon\int_{\mathbb{B}}[~ u_{t}^{2}  -
|\nabla_{\mathbb{B}} u|^{2}~] \frac{dx_{1}}{x_{1}}dx'
\nonumber\\&+&\epsilon\int_{\mathbb{B}}h(x)|u|^{p}\frac{dx_{1}}{x_{1}}dx'
~-~ \epsilon\int_{\mathbb{B}}f(x)|u_{t}|^{m-2}u_{t} u
\frac{dx_{1}}{x_{1}}dx' +
\epsilon\int_{0}^{t}g(t-\tau)\|\nabla_{\mathbb{B}}u\|^{2}_{L^{\frac{n}{2}}(\mathbb{B})}d\tau\nonumber
\\&-&\epsilon\int_{0}^{t}g(t-\tau)
\|\nabla_{\mathbb{B}}u\|_{L^{\frac{n}{2}}(\mathbb{B})}
~\|\nabla_{\mathbb{B}} u(\tau) - \nabla_{\mathbb{B}}
u(t)\|_{L^{\frac{n}{2}}(\mathbb{B})} d\tau.
\end{eqnarray}
On the other hand, we utilize Young's inequality to estimate the
last term on the right hand side of \ref{2.32} and use the
definition of $I(t)$ to substitue for
$\int_{\mathbb{B}}h(x)|u|^{p}\frac{dx_{1}}{x_{1}}dx'.$ Hence,
\begin{eqnarray}\label{2.33}
\mathcal{F}^{'}(t) &\geq&
(1-k)\mathcal{I}^{-k}(t)\int_{\mathbb{B}}f(x)|u_{t}|^{m}
\frac{dx_{1}}{x_{1}}dx' + \epsilon\int_{\mathbb{B}}u_{t}^{2}
\frac{dx_{1}}{x_{1}}dx' ~-~\epsilon\biggl(1 - \int_{0}^{t}
g(s)ds\biggr)\|\nabla_{\mathbb{B}}
u\|^{2}_{L^{\frac{n}{2}}(\mathbb{B})}\nonumber\\&+&\epsilon\biggl(
\frac{p}{2}\|u_{t}\|^{2}_{L^{\frac{n}{2}}(\mathbb{B})} +
\frac{p}{2}(g~o\nabla_{\mathbb{B}} u)(t) + \frac{p}{2}[1 -
\int_{0}^{t} g(s)ds] \|\nabla_{\mathbb{B}}
u\|^{2}_{L^{\frac{n}{2}}(\mathbb{B})} - p
I(t)\biggr)\nonumber\\&-&\epsilon\int_{\mathbb{B}}f(x)|u_{t}|^{m-2}u_{t}
u \frac{dx_{1}}{x_{1}}dx' - \epsilon\tau(g~o\nabla_{\mathbb{B}}u)(t)
-\frac{\epsilon}{4\tau}\int_{0}^{t}g(s)ds\|\nabla_{\mathbb{B}}u(t)\|^{2}_{L^{\frac{n}{2}}(\mathbb{B})}\nonumber
\\&\geq& (1-k)\mathcal{I}^{-k}(t)\int_{\mathbb{B}}f(x)|u_{t}|^{m}
\frac{dx_{1}}{x_{1}}dx' + \epsilon\int_{\mathbb{B}}u_{t}^{2}
\frac{dx_{1}}{x_{1}}dx' ~-~\epsilon\biggl(1 - \int_{0}^{t}
g(s)ds\biggr)\|\nabla_{\mathbb{B}}
u\|^{2}_{L^{\frac{n}{2}}(\mathbb{B})}\nonumber\\&+&\epsilon\biggl(
\frac{p}{2}\|u_{t}\|^{2}_{L^{\frac{n}{2}}(\mathbb{B})} +
\frac{p}{2}(g~o\nabla_{\mathbb{B}} u)(t) + \frac{p}{2}[1 -
\int_{0}^{t} g(s)ds] \|\nabla_{\mathbb{B}}
u\|^{2}_{L^{\frac{n}{2}}(\mathbb{B})} +p\mathcal{I}(t) -
pI_{1}\biggr)\nonumber\\&-&\epsilon\int_{\mathbb{B}}f(x)|u_{t}|^{m-2}u_{t}
u \frac{dx_{1}}{x_{1}}dx' - \epsilon\tau(g~o\nabla_{\mathbb{B}}u)(t)
-\frac{\epsilon}{4\tau}\int_{0}^{t}g(s)ds\|\nabla_{\mathbb{B}}u(t)\|^{2}_{L^{\frac{n}{2}}(\mathbb{B})}
\nonumber\\&\geq&
(1-k)\mathcal{I}^{-k}(t)\int_{\mathbb{B}}f(x)|u_{t}|^{m}
\frac{dx_{1}}{x_{1}}dx' ~+~\epsilon(1 +
\frac{p}{2})\int_{\mathbb{B}}u_{t}^{2}~\frac{dx_{1}}{x_{1}}dx'~+~\epsilon
p \mathcal{I}(t)\nonumber\\&+&\epsilon(\frac{p}{2} -
\tau)~(g~o\nabla_{\mathbb{B}} u)(t) -
\epsilon\int_{\mathbb{B}}f(x)|u_{t}|^{m-2} u_{t} u
~\frac{dx_{1}}{x_{1}}dx' \nonumber\\&+&\epsilon\biggl[(\frac{p}{2} -
1) - (\frac{p}{2} - 1 +
\frac{1}{4\tau})\int_{0}^{\infty}g(s)ds\biggr]\|\nabla_{\mathbb{B}}
u\|^{2}_{L^{\frac{n}{2}}(\mathbb{B})},
\end{eqnarray}
for some $0 < \tau < \frac{p}{2}.$

Therefore,
\begin{eqnarray}\label{2.34}
\mathcal{F}^{'}(t)&\geq&
(1-k)\mathcal{I}^{-k}(t)\int_{\mathbb{B}}f(x)|u_{t}|^{m}
\frac{dx_{1}}{x_{1}}dx' ~+~\epsilon(1 +
\frac{p}{2})\int_{\mathbb{B}}u_{t}^{2}~\frac{dx_{1}}{x_{1}}dx'~+~\epsilon
p \mathcal{I}(t)\nonumber\\&+&\epsilon M_{1}(g~o\nabla_{\mathbb{B}}
u)(t)  + \epsilon M_{2}\|\nabla_{\mathbb{B}}
u\|^{2}_{L^{\frac{n}{2}}(\mathbb{B})} -
\epsilon\int_{\mathbb{B}}f(x)|u_{t}|^{m-2}u_{t} u
~\frac{dx_{1}}{x_{1}}dx',
\end{eqnarray}
where $M_{1}:= \frac{p}{2} - \tau$ and $M_{2}:= (\frac{p}{2} - 1) -
(\frac{p}{2} - 1 + \frac{1}{4\tau})\int_{0}^{\infty}g(s)ds$ are
positive constants.

To estimate of the last term in \ref{2.34}, we exploit Young's
inequality as follows:
\begin{eqnarray}\label{2.35}
\int_{\mathbb{B}}f(x)|u_{t}|^{m-2} u_{t} u ~\frac{dx_{1}}{x_{1}}dx'
&\leq& \int_{\mathbb{B}}|f(x)| ~|u|~
|u_{t}|^{m-1}\frac{dx_{1}}{x_{1}}dx'
\nonumber\\&\leq&\biggl[\frac{\theta^{m}C_{f}^{m}}{m}\int_{\mathbb{B}}
|u|^{m}~\frac{dx_{1}}{x_{1}}dx' +
\frac{m-1}{m}\theta^{-\frac{m}{m-1}}\int_{\mathbb{B}}|u_{t}|^{m}~\frac{dx_{1}}{x_{1}}dx'\biggr]
\nonumber\\&\leq&
\biggl[\frac{\theta^{m}C_{f}^{m}}{m}\|u\|^{m}_{L^{\frac{n}{m}}(\mathbb{B})}
+
\frac{m-1}{m}\theta^{-\frac{m}{m-1}}\|u_{t}\|^{m}_{L^{\frac{n}{m}}(\mathbb{B})}\biggr].
\end{eqnarray}
Now, we insert \ref{2.35} in \ref{2.34}, then we get
\begin{eqnarray}\label{2.36}
\mathcal{F}^{'}(t) &\geq&
(1-k)\mathcal{I}^{-k}(t)\int_{\mathbb{B}}f(x)|u_{t}|^{m}\frac{dx_{1}}{x_{1}}dx'
+ \epsilon(1 +
\frac{p}{2})\int_{\mathbb{B}}u^{2}_{t}(x,t)\frac{dx_{1}}{x_{1}}dx'
\nonumber\\&+& \epsilon p \mathcal{I}(t) + \epsilon
M_{1}(g~o~\nabla_{\mathbb{B}}u)(t)  + \epsilon
M_{2}\|\nabla_{\mathbb{B}} u\|^{2}_{L^{\frac{n}{2}}(\mathbb{B})} -
\frac{\epsilon \theta^{m}
C_{f}^{m}}{m}\|u\|^{m}_{L^{\frac{n}{m}}(\mathbb{B})} \nonumber\\&-&
\frac{\epsilon
(m-1)\theta^{-\frac{m}{m-1}}}{m}\|u_{t}\|^{m}_{L^{\frac{n}{m}}(\mathbb{B})}.
\end{eqnarray}
Sine our itegral is taken over the variable $x=(x_{1},x^{'}),$ we
can consider the parametre $\theta.$ Hence, we consider
$\theta^{-\frac{m}{m-1}} = N ~\mathcal{I}^{-k}(t)$ for large enough
$N.$ Now, we apply this equality in relation \ref{2.36} and by
relation \ref{1}  obtain
\begin{eqnarray}\label{2.37}
\mathcal{F}^{'}(t) &\geq& (1-k)\mathcal{I}^{-k}(t) C^{*}
\|\nabla_{\mathbb{B}} u_{t}\|^{m}_{L^{\frac{n}{m}}(\mathbb{B})} +
\epsilon(1 +
\frac{p}{2})\int_{\mathbb{B}}u^{2}_{t}(x,t)\frac{dx_{1}}{x_{1}}dx'
\nonumber\\&+& \epsilon p \mathcal{I}(t) + \epsilon
M_{1}(g~o~\nabla_{\mathbb{B}}u)(t)  + \epsilon
M_{2}\|\nabla_{\mathbb{B}} u\|^{2}_{L^{\frac{n}{2}}(\mathbb{B})} -
~\frac{\epsilon N^{1-m}~C_{f}^{m}}{m}\mathcal{I}^{k(m-1)}(t)
\|u\|^{m}_{L^{\frac{n}{m}}(\mathbb{B})} \nonumber\\&-&
\frac{\epsilon
(m-1)N}{m}\mathcal{I}^{-k}(t)\|u_{t}\|^{m}_{L^{\frac{n}{m}}(\mathbb{B})}
= (1-k)\mathcal{I}^{-k}(t) C^{*} C_{Poin}^{m}\|
u_{t}\|^{m}_{L^{\frac{n}{m}}(\mathbb{B})} \nonumber\\&+& \epsilon(1
+ \frac{p}{2})\int_{\mathbb{B}}u^{2}_{t}(x,t)\frac{dx_{1}}{x_{1}}dx'
\nonumber\\&+& \epsilon p \mathcal{I}(t) + \epsilon
M_{1}(g~o~\nabla_{\mathbb{B}}u)(t)  + \epsilon
M_{2}\|\nabla_{\mathbb{B}} u\|^{2}_{L^{\frac{n}{2}}(\mathbb{B})} -
~\frac{\epsilon N^{1-m}~C_{f}^{m}}{m}\mathcal{I}^{k(m-1)}(t)
\|u\|^{m}_{L^{\frac{n}{m}}(\mathbb{B})} \nonumber\\&-&
\frac{\epsilon
(m-1)N}{m}\mathcal{I}^{-k}(t)\|u_{t}\|^{m}_{L^{\frac{n}{m}}(\mathbb{B})}
\nonumber\\&=&\biggr[(1-k)C^{*}C_{Poin}^{m} - \frac{\epsilon N
(m-1)}{m}\biggr]\mathcal{I}^{-k}(t)\|u_{t}\|^{m}_{L^{\frac{n}{m}}(\mathbb{B})}
+ \epsilon(1 +
\frac{p}{2})\int_{\mathbb{B}}u^{2}_{t}(x,t)\frac{dx_{1}}{x_{1}}dx'
\nonumber\\&+& \epsilon ~p~ \mathcal{I}(t) ~+~ \epsilon~
M_{1}(g~o~\nabla_{\mathbb{B}}u)(t)  ~+~ \epsilon~
M_{2}\|\nabla_{\mathbb{B}}
u\|^{2}_{L^{\frac{n}{2}}(\mathbb{B})}\nonumber\\&+&
\epsilon\biggl[p~\mathcal{I}(t)  ~-~ \frac{N^{1-m}C_{f}^{m}}{m}
\mathcal{I}^{k(m-1)}(t)
\|u\|^{m}_{L^{\frac{n}{m}}(\mathbb{B})}\biggr].
\end{eqnarray}
Now, we utilize relation \ref{2.28} and embedding inequality
$\|u\|^{m}_{L^{\frac{n}{m}}(\mathbb{B})} \leq
C_{emb}\|u\|^{m}_{L^{\frac{n}{p}}(\mathbb{B})},$
\begin{equation*}
\mathcal{I}^{k(m-1)}(t)\|u\|^{m}_{L^{\frac{n}{m}}(\mathbb{B})} \leq
(\frac{C_{h}}{p})^{k(m-1)}C_{emb}\|u\|^{m+kp(m-1)}_{L^{\frac{n}{p}}(\mathbb{B})}.
\end{equation*}
Therefore, it follows from \ref{2.37} and Lemma \ref{l.3} for $s = m
+ kp(m-1) \leq p,$
\begin{eqnarray}\label{2.38}
\mathcal{F}^{'}(t)&\geq& \biggr[(1-k)C^{*}C_{Poin}^{m} -
\frac{\epsilon N
(m-1)}{m}\biggr]\mathcal{I}^{-k}(t)\|u_{t}\|^{m}_{L^{\frac{n}{m}}(\mathbb{B})}
+ \epsilon(1 +
\frac{p}{2})\int_{\mathbb{B}}u^{2}_{t}(x,t)\frac{dx_{1}}{x_{1}}dx'
\nonumber\\&+& \epsilon ~p~ \mathcal{I}(t) ~+~ \epsilon~
M_{1}(g~o~\nabla_{\mathbb{B}}u)(t)  ~+~ \epsilon~
M_{2}\|\nabla_{\mathbb{B}}
u\|^{2}_{L^{\frac{n}{2}}(\mathbb{B})}\nonumber\\&+&
\epsilon\biggl[p~\mathcal{I}(t)
~-~\frac{N^{1-m}~C_{f}^{m}}{m}\mathcal{I}^{k(m-1)}(t)\|u\|^{m}_{L^{\frac{n}{m}}(\mathbb{B})}
\biggr] \nonumber\\&=&\biggr[(1-k)C^{*}C_{Poin}^{m} - \frac{\epsilon
N
(m-1)}{m}\biggr]\mathcal{I}^{-k}(t)\|u_{t}\|^{m}_{L^{\frac{n}{m}}(\mathbb{B})}
+ \epsilon(1 +
\frac{p}{2})\int_{\mathbb{B}}u^{2}_{t}(x,t)\frac{dx_{1}}{x_{1}}dx'
\nonumber\\&+& \epsilon\biggl[p~\mathcal{I}(t)  ~-~
\frac{N^{1-m}C_{f}^{m}~C_{emb}}{m}~(\frac{C_{h}}{p})^{k(m-1)}
\biggl\{-2\mathcal{I}(t) -
\|u_{t}\|^{2}_{L^{\frac{n}{2}}(\mathbb{B})} \nonumber\\&-&~
(g~o~\nabla_{\mathbb{B}} u)(t)
~+~\|u\|^{p}_{L^{\frac{n}{p}}(\mathbb{B})}\biggr\}\biggr]\nonumber\\&\geq&
\biggr[(1-k)C^{*}C_{Poin}^{m} - \frac{\epsilon N
(m-1)}{m}\biggr]\mathcal{I}^{-k}(t)\|u_{t}\|^{m}_{L^{\frac{n}{m}}(\mathbb{B})}
+ \epsilon(1 + \frac{p}{2} +
R)\|u_{t}\|^{2}_{L^{\frac{n}{2}}(\mathbb{B})} \nonumber\\&+&
\epsilon (~p~+ 2R) \mathcal{I}(t) ~+~ \epsilon~ (M_{1} +
R)(g~o~\nabla_{\mathbb{B}}u)(t)  \nonumber\\&+&~ \epsilon~
M_{2}\|\nabla_{\mathbb{B}} u\|^{2}_{L^{\frac{n}{2}}(\mathbb{B})} -
\epsilon R\|u\|^{p}_{L^{\frac{n}{p}}(\mathbb{B})},
\end{eqnarray}
where $R :=
\frac{N^{1-m}~C_{f}^{m}~C_{emb}}{m}~(\frac{C_{h}}{p})^{k(m-1)}$ is a
positive constant. We set $M-{3} < \min\{M_{1},M_{2},
\frac{p}{2}\}.$ Moreover, we can get the following estimations for
$p= 2M_{3} + (p - 2M_{3}):$
\begin{eqnarray}\label{2.39}
\mathcal{F}^{'}(t) &\geq& \biggr[(1-k)C^{*}C_{Poin}^{m} -
\frac{\epsilon N
(m-1)}{m}\biggr]\mathcal{I}^{-k}(t)\|u_{t}\|^{m}_{L^{\frac{n}{m}}(\mathbb{B})}
+ \epsilon(1 + \frac{p}{2} + R -
M_{3})\|u_{t}\|^{2}_{L^{\frac{n}{2}}(\mathbb{B})} \nonumber\\&+&
\epsilon (~p~+ 2R -2M_{3}) \mathcal{I}(t) ~+~ \epsilon~ (M_{1} + R -
M_{3})(g~o~\nabla_{\mathbb{B}}u)(t)  \nonumber\\&+&~ \epsilon~
(M_{2} - M_{3})\|\nabla_{\mathbb{B}}
u\|^{2}_{L^{\frac{n}{2}}(\mathbb{B})} - \epsilon (\frac{2M_{3}}{p}
-R)\|u\|^{p}_{L^{\frac{n}{p}}(\mathbb{B})}.
\end{eqnarray}
For large enough $N,$ we take $$\gamma:= \min\biggl\{(1 +
\frac{p}{2} + R - M_{3}),~(~p~+ 2R -2M_{3}),~(M_{1} + R -
M_{3}),(M_{2} - M_{3})\biggr\}.$$ Then,
\begin{eqnarray}\label{2.40}
\mathcal{F}^{'}(t) &\geq& \biggr[(1-k)C^{*}C_{Poin}^{m} -
\frac{\epsilon N
(m-1)}{m}\biggr]\mathcal{I}^{-k}(t)\|u_{t}\|^{m}_{L^{\frac{n}{m}}(\mathbb{B})}\nonumber
\\&+&\epsilon \gamma\biggl[\mathcal{I}(t) +
\|u_{t}\|^{2}_{L^{\frac{n}{2}}(\mathbb{B})} +
\|u\|^{p}_{L^{\frac{n}{p}}(\mathbb{B})}  + (g~o~\nabla_{\mathbb{B}}
u)(t)\biggr].
\end{eqnarray}
On the other hand, sine $\mathcal{I}(t)\geq \mathcal{I}(0) > 0,$ it
follows that $$\mathcal{F}(t) \geq \mathcal{F}(0) > 0~~~~\forall
t\geq 0.$$Therefore,
\begin{eqnarray}\label{2.41}
\mathcal{F}^{'}(t) \geq \epsilon \gamma\biggl[\mathcal{I}(t) +
\|u_{t}\|^{2}_{L^{\frac{n}{2}}(\mathbb{B})} +
\|u\|^{p}_{L^{\frac{n}{p}}(\mathbb{B})}  + (g~o~\nabla_{\mathbb{B}}
u)(t)\biggr].
\end{eqnarray}
Furthermore, by H\"{o}lder's inequality and embedding map one can
get
\begin{eqnarray}\label{2.42}
\biggl|\int_{\mathbb{B}}u~u_{t}\frac{dx_{1}}{x_{1}}dx'\biggr| \leq
\|u\|_{L^{\frac{n}{2}}(\mathbb{B})}~\|u_{t}\|_{L^{\frac{n}{2}}(\mathbb{B})}
\leq
C_{emb}~\|u\|_{L^{\frac{n}{p}}(\mathbb{B})}~\|u_{t}\|_{L^{\frac{n}{2}}(\mathbb{B})}.
\end{eqnarray}
Hence,
\begin{eqnarray}\label{2.43}
\biggl|\int_{\mathbb{B}}u~u_{t}\frac{dx_{1}}{x_{1}}dx'\biggr|^{\frac{1}{1-k}}
\leq
C_{emb}^{\frac{1}{1-k}}~\|u\|^{\frac{1}{1-k}}_{L^{\frac{n}{p}}(\mathbb{B})}~\|u_{t}\|^{\frac{1}{1-k}}_{L^{\frac{n}{2}}(\mathbb{B})}.
\end{eqnarray}
We exploit Young's inequality, thus
\begin{eqnarray}\label{2.44}
\biggl|\int_{\mathbb{B}}u~u_{t}\frac{dx_{1}}{x_{1}}dx'\biggr|^{\frac{1}{1-k}}
\leq
C_{emb}^{\frac{1}{1-k}}\biggl(\frac{D^{\frac{\mu}{1-k}}}{\mu}\|u\|^{\frac{\mu}{1-k}}_{L^{\frac{n}{p}}(\mathbb{B})}
+
\frac{D^{-\frac{\nu}{1-k}}}{\nu}\|u_{t}\|^{\frac{\nu}{1-k}}_{L^{\frac{n}{2}}(\mathbb{B})}\biggr),
\end{eqnarray}
such that $\frac{1}{\mu} + \frac{1}{\nu} = 1.$ In fact, we consider
$\frac{\mu}{1-k} = \frac{2}{(1-2k)} \leq p,$ then $\nu = 2(1-k).$

Therefore,
\begin{eqnarray}\label{2.45}
\biggl|\int_{\mathbb{B}}u~u_{t}\frac{dx_{1}}{x_{1}}dx'\biggr|^{\frac{1}{1-k}}
\leq A\biggr\{\|u\|^{s}_{L^{\frac{n}{p}}(\mathbb{B})} +
\|u_{t}\|^{2}_{L^{\frac{n}{2}}(\mathbb{B})}\biggr\},
\end{eqnarray}
where $s= \frac{2}{(1-2k)} \leq p$ and $A:=
\max\{C_{emb}^{\frac{1}{1-k}}D^{\frac{\mu}{1-k}}
,C_{emb}^{\frac{1}{1-k}} D^{-\frac{\nu}{1-k}}\}.$ Again, we apply
Lemma \ref{l.3},
\begin{eqnarray}\label{2.46}
\biggl|\int_{\mathbb{B}}u~u_{t}\frac{dx_{1}}{x_{1}}dx'\biggr|^{\frac{1}{1-k}}
&\leq&
A\biggl\{(1-\Lambda)\|u_{t}\|^{2}_{L^{\frac{n}{2}}(\mathbb{B})} +
\Lambda\|u\|^{p}_{L^{\frac{n}{p}}(\mathbb{B})} -
\Lambda~(g~o~\nabla_{\mathbb{B}}u)(t) -
2\Lambda\mathcal{I}(t)\biggr\}\nonumber\\&\leq&
G\biggr\{\mathcal{I}(t) +
\|u_{t}\|^{2}_{L^{\frac{n}{2}}(\mathbb{B})} +
(g~o~\nabla_{\mathbb{B}}u)(t) +
\|u\|^{p}_{L^{\frac{n}{p}}(\mathbb{B})}\biggr\},
\end{eqnarray}
Where $G:= \max\{A(1-\Lambda), 2A\Lambda\}$ is a positive constant.
Then, for every $t \geq 0$ one can obtain
\begin{eqnarray}\label{2.47}
\mathcal{F}^{\frac{1}{1-k}}(t) &=& \biggl(\mathcal{I}^{1-k}(t) +
\epsilon\int_{\mathbb{B}}
u_{t}(x,t)~u(x,t)\frac{dx_{1}}{x_{1}}dx'\biggr)^{\frac{1}{1-k}}\nonumber\\&\leq&
2^{\frac{1}{1-k}}\biggl(\mathcal{I}(t) +
\biggl|\int_{\mathbb{B}}u~u_{t}\frac{dx_{1}}{x_{1}}dx'\biggr|^{\frac{1}{1-k}}\biggr)
\nonumber\\&\leq& 2^{\frac{1}{1-k}}\Biggl(\mathcal{I}(t) +
G\biggr\{\mathcal{I}(t) +
\|u_{t}\|^{2}_{L^{\frac{n}{2}}(\mathbb{B})} +
(g~o~\nabla_{\mathbb{B}}u)(t) +
\|u\|^{p}_{L^{\frac{n}{p}}(\mathbb{B})}\biggr\}\Biggr)\nonumber\\&\leq&
\Gamma \biggr\{\mathcal{I}(t) +
\|u_{t}\|^{2}_{L^{\frac{n}{2}}(\mathbb{B})} +
(g~o~\nabla_{\mathbb{B}}u)(t) +
\|u\|^{p}_{L^{\frac{n}{p}}(\mathbb{B})}\biggr\},
\end{eqnarray}
where $\Gamma = \max\{2^{\frac{1}{1-k}} , 2^{\frac{1}{1-k}}G\}$ is a
positive constant. One can use relations \ref{2.41} and \ref{2.47}
and then obtains for any $t \geq 0 $ that
\begin{eqnarray}\label{2.48}
\mathcal{F}^{'}(t) \geq \Omega~\mathcal{F}^{\frac{1}{1-k}}(t),
\end{eqnarray}
where $\Omega$ is a positive constan depending only $\epsilon\gamma$
and $\Gamma.$ Now, we take integral from \ref{2.48} over interval
$(0 , t)$ and get
\begin{eqnarray}\label{2.49}
\mathcal{F}^{\frac{1}{1-k}}(t) \geq
\frac{1}{\mathcal{F}^{-\frac{k}{1-k}}(0) -  \Omega t \frac{k}{1-k}}.
\end{eqnarray} Hence, it follows from \ref{2.49} that
$\mathcal{F}(t)$ blows up in time
\begin{eqnarray}\label{2.50}
T(k,\Omega) = T \leq \frac{1-k}{\Omega k
(\mathcal{F}(0))^\frac{k}{1-k}}.
\end{eqnarray}
\end{proof}

 Here, we investigate the lower bound of the blow up time for the
blow up solution of problem \ref{1.1}.

{\bf Proof of Theorem \ref{th.2}}
\begin{proof}
First, we define $\mathcal{H}(t) =
\int_{\mathbb{B}}h(x)|u(x,t)|^{p}\frac{dx_{1}}{x_{1}}dx'.$ Then,
\begin{eqnarray}\label{2.51}
\mathcal{H}^{'}(t) &=& p
\int_{\mathbb{B}}|u(x,t)|^{p-2}~u(x,t)u_{t}(x,t)\frac{dx_{1}}{x_{1}}dx'
\nonumber\\&\leq&
\frac{p~C_{h}}{2}\Bigg(\int_{\mathbb{B}}|u(x,t)|^{2(p-1)}\frac{dx_{1}}{x_{1}}dx'
+  \int_{\mathbb{B}}|u_{t}(x,t)|^{2}\frac{dx_{1}}{x_{1}}dx'\Biggr)
\end{eqnarray}
To estimate the first term on the right side of inequality
\ref{2.51}, we consider the following two cases. In the first case,
we consider $2 < p \leq 2^{*}.$ Suppose that $q = 2(p-1)$ and $r =
n(p-2).$ Using of H\"{o}lder's inequality and embeding Theorem, one
can get
\begin{equation*}
\int_{\mathbb{B}}|u(x,t)|^{q}\frac{dx_{1}}{x_{1}}dx' =
\int_{\mathbb{B}}|u|^{q\eta}~|u|^{q(1-\eta)}\frac{dx_{1}}{x_{1}}dx'
\leq\Biggl(\int_{\mathbb{B}}|u|^{r}\frac{dx_{1}}{x_{1}}dx'\Biggr)^{\frac{q\eta}{r}}\Biggl(\int_{\mathbb{B}}|u|^{2^{*}}\frac{d
x_{1}}{x_{1}}dx'\Biggr)^{\frac{q(1-\eta)}{2^{*}}},
\end{equation*}
where $\eta$ satisfies $\frac{q\eta}{r} + \frac{q(1-\eta)}{2^{*}} =
1.$ A simple calculation shows that
$$\eta = \frac{r(2^{*} - q)}{q(2^{*} - r)},\hspace{1 cm} \frac{q\eta}{r} = \frac{2}{n}~,\hspace{1 cm}
\frac{q - q\eta}{2^{*}} = 1 - \frac{2}{n}.$$ Then, we use the
H\"{o}lder inequality
\[\|u\|^{\frac{2r}{n}}_{L^{\frac{r}{n}}(\mathbb{B})} \leq |\mathbb{B}^{\frac{2(p - r)}{np}}\|u\|
^{\frac{2r}{n}}_{L^{\frac{n}{p}}(\mathbb{B})}\leq \bigg(1 +
|\mathbb{B}^{\frac{2(p - r)}{np}}\biggr)\|u\|
^{\frac{2r}{n}}_{L^{\frac{n}{p}}(\mathbb{B})}\] and embedding
inequality $\|u\|_{L^{\frac{n}{2^{*}}}(\mathbb{B})} \leq
\mathcal{C}_{*}\|\nabla_{\mathbb{B}}u\|_{L^{\frac{n}{2}}(\mathbb{B})},$
where $\mathcal{C}_{*}$ is the best constant of the Sobolev
embedding
$\mathcal{H}^{1,\frac{n}{2}}_{2,0}(\mathbb{B})\hookrightarrow
L^{\frac{n}{2^{*}}}(\mathbb{B}).$
\begin{eqnarray}\label{2.52}
\|u\|^{q}_{L^{\frac{n}{q}}(\mathbb{B})}
&\leq&\|u\|^{q\eta}_{L^{\frac{n}{r}}(\mathbb{B})}~\|u\|^{q(1-\eta)}_{L^{\frac{n}{2^{*}}}(\mathbb{B})}=
\|u\|^{\frac{2r}{n}}_{L^{\frac{n}{r}}(\mathbb{B})}\|u\|^{2}_{L^{\frac{n}{2^{*}}}(\mathbb{B})}\nonumber
\\&\leq& \mathcal{C}_{*}^{2}\biggl(1 + |\mathbb{B}|^{\frac{2(p-
r)}{np}}\biggr)~\|u\|^{\frac{2r}{n}}_{L^{\frac{n}{p}}(\mathbb{B})}~\|\nabla_{\mathbb{B}}u\|^{2}_{L^{\frac{n}{2}}(\mathbb{B})}\nonumber
\\&\leq& \mathcal{C}_{*}^{2}\biggl(1 + |\mathbb{B}|^{\frac{2(p-
r)}{np}}\biggr)~\biggl[\|u\|^{\frac{2r
.s}{n}}_{L^{\frac{n}{p}}(\mathbb{B})}~+~\|\nabla_{\mathbb{B}}u\|^{2t}_{L^{\frac{n}{2}}(\mathbb{B})}\biggr]
\nonumber\\&\leq&
\mathcal{C}_{1}\Biggl(\|u\|^{p}_{L^{\frac{n}{p}}(\mathbb{B})} +
\|\nabla_{\mathbb{B}}u\|^{2}_{L^{\frac{n}{2}}(\mathbb{B})}\Biggr)^{k_{1}},
\end{eqnarray}
where $\frac{1}{s} + \frac{1}{t} = 1,$ $t:= \frac{2r}{np}.s$ and we
can deduce $k_{1}= \frac{3p -4}{p}$ and $\mathcal{C}_{1}=
\mathcal{C}_{*}^{2}\biggl(1 + |\mathbb{B}|^{\frac{2(p-
r)}{np}}\biggr).$

For the second case, we assume that $\frac{2n}{n-2} < p \leq
\frac{2(n-1)}{n-2}.$ Then,

\begin{eqnarray}\label{2.53}
\|u\|^{q}_{L^{\frac{n}{r}}(\mathbb{B})} &\leq&
\mathcal{C}_{*}^{r}\biggl(1 + |\mathbb{B}|^{1 -
\frac{q}{2^{*}}}\biggr)\|\nabla_{\mathbb{B}}u\|^{q}_{L^{\frac{n}{2}}(\mathbb{B})}\leq
\mathcal{C}_{*}^{r}\biggl(1 + |\mathbb{B}|^{1 -
\frac{q}{2^{*}}}\biggr)\Biggl[
\|\nabla_{\mathbb{B}}u\|^{q}_{L^{\frac{n}{2}}(\mathbb{B})} +
\|u\|^{p(p-1)}_{L^{\frac{n}{p}}(\mathbb{B})}\Biggr]\nonumber\\&\leq&
\mathcal{C}_{2}\Biggl(\|u\|^{p}_{L^{\frac{n}{p}}(\mathbb{B})} +
\|\nabla_{\mathbb{B}}u\|^{2}_{L^{\frac{n}{2}}(\mathbb{B})}\Biggr)^{k_{2}},
\end{eqnarray}
where $k_{2}= p-1,$ and $\mathcal{C}_{2} =
\mathcal{C}_{*}^{r}\biggl(1 + |\mathbb{B}|^{1 -
\frac{q}{2^{*}}}\biggr)$ is a positive constant.

From Lemma \ref{l.1}
\begin{eqnarray}\label{2.54}
I(t) \leq I(0)  &=& \frac{1}{2}
\|u_{1}\|^{2}_{L^{\frac{n}{2}}(\mathbb{B})} +
\|\nabla_{\mathbb{B}}u_{0}\|_{L^{\frac{n}{2}}(\mathbb{B})}^{2}
\nonumber\\&-& \frac{1}{p} \int_{\mathbb{B}}
h(x)|u|^{p}\frac{dx_{1}}{x_{1}}dx' \hspace{1 cm} \forall t\in [0~ ,~
T(k,\Omega)= T~),
\end{eqnarray}
where $k$ and $\Omega$ are given in proof of Theorem \ref{th.1}.
Now, using of the definitopn of the functional $I(t),$ assumption
$A_{3}$ and inequality \ref{2.53} we obtain
\begin{eqnarray}\label{2.55}
\|u_{t}(x,t)\|^{2}_{L^{\frac{n}{2}}(\mathbb{B})} &+&
\frac{1}{(p-1)^{2}}\|\nabla_{\mathbb{B}}
u\|^{2}_{L^{\frac{n}{2}}(\mathbb{B})}~ +~ (g~o~u)(t)
\nonumber\\&\leq&\|u_{t}(x,t)\|^{2}_{L^{\frac{n}{2}}(\mathbb{B})} +
\Bigg(1 - \int_{0}^{t} g(s)ds\Biggr)\|\nabla_{\mathbb{B}}
u\|^{2}_{L^{\frac{n}{2}}(\mathbb{B})}~ +~
(g~o~u)(t)\nonumber\\&=&\frac{2}{p}\|u_{t}\|^{p}_{L^{\frac{n}{p}}(\mathbb{B})}
+ 2I(t) \leq \frac{2}{p} \mathcal{H}(t) +  2I(0).
\end{eqnarray}
From computaions \ref{2.51}-\ref{2.55}, we obtain the following
inequality:
\begin{eqnarray}\label{2.56}
\mathcal{H}^{'}(t) &\leq&
\frac{C_{h}p}{2}\Biggl(\mathcal{C}_{i}\biggl[\|u\|^{p}_{L^{\frac{n}{p}}(\mathbb{B})}
+ \|\nabla_{\mathbb{B}}u\|^{2}_{L^{\frac{n}{2}}(\mathbb{B})}\biggr]
+~ \frac{2}{p}\mathcal{H}(t) + 2I(0)\Biggr)\nonumber\\&\leq&
\frac{C_{h}p}{2}\Biggl(\mathcal{C}_{i}\biggl[\mathcal{H}(t) +
\mathcal{C}_{0}\biggl(\frac{2}{p}\mathcal{H}(t) +
2I(0)\biggr)\biggr]^{k_{i}} + \frac{2}{p}\mathcal{H}(t) +
2I(0)\Biggr)\nonumber\\&\leq&
\frac{C_{h}p}{2}\Biggl(\mathcal{C}_{i}\biggl[\biggr(1 +
\frac{2\mathcal{C}_{0}}{p}\biggr)\mathcal{H}(t) + 2\mathcal{C}_{0}
I(0)\bigg]^{k_{i}} + \frac{2}{p}\mathcal{H}(t) +
2I(0)\Biggr)\nonumber\\&\leq&
\frac{C_{h}p}{2}~2^{k_{i}-1}\Biggl(\biggl[1 +
\frac{2\mathcal{C}_{0}}{p}\biggr]^{k_{i}}\mathcal{H}^{k_{i}}(t) +
\biggl(2\mathcal{C}_{0}I(0)\biggr)^{k_{i}}\Biggr) + \mathcal{H}(t) +
pI(0) \nonumber\\&=& \mathcal{C}_{3}\mathcal{H}^{k_{i}}(t) +
\mathcal{H}(t) + \mathcal{C}_{4},
\end{eqnarray}
where $\mathcal{C}_{0}= \frac{1}{ l},$ $\mathcal{C}_{3}=
\frac{C_{h}p\mathcal{C}_{i}2^{k_{i}}}{4}\biggl(1 +
\frac{2\mathcal{C}_{0}}{p}\biggr)^{k_{i}}$ and $\mathcal{C}_{4} =
pI(0) +
\frac{C_{h}p\mathcal{C}_{i}}{4}\biggl(4\mathcal{C}_{0}I(0)\biggr)^{k_{i}}$
for $i=1,2$ are positive constants. We exploit  Theorem \ref{th.1}
and then get
\begin{eqnarray}\label{2.57}
\lim\limits_{t\rightarrow T}\int_{\mathbb{B}}
h(x)|u(x,t)|^{p}\frac{dx_{1}}{x_{1}}dx' = + \infty.
\end{eqnarray}
It follows from \ref{2.56} and \ref{2.57},
\begin{eqnarray}\label{2.58}
\int_{\mathcal{H}(0)}^{\infty} \frac{1}{\mathcal{C}_{3}S^{k} + S +
\mathcal{C}_{4}}ds \leq T.
\end{eqnarray}
\end{proof}

\end{document}